\documentclass[final,leqno]{siamltex704}
\usepackage{hyperref}
\usepackage{amsmath}
\usepackage{graphicx}
\usepackage{array}
\usepackage{lineno}

\usepackage{tikz}
\usepackage{amsfonts,amssymb}
\usepackage{dsfont}
\usepackage{pifont}
\def\E{{\mathcal E}}

\def\O{\Omega}
\def\pa{\partial}
\newcommand{\bq}{{\bf q}}
\newcommand{\bn}{{\bf n}}
\newcommand{\bx}{{\bf x}}

\def\T{{\mathcal T}}

\def\W{{\mathcal W}}

\def\l{{\langle}}
\def\r{{\rangle}}

\def\pa{\partial}
\def\he{{\widehat e}}
\def\hQ{{\widehat Q}}
\def\hw{{\widehat w}}

\def\O{\Omega}

\def\bbQ{\mathbb{Q}}
\newcommand{\pT}{{\partial T}}
\def\3bar{{|\hspace{-.02in}|\hspace{-.02in}|}}

\def\3bar{{|\hspace{-.02in}|\hspace{-.02in}|}}

\def\ad#1{\begin{aligned}#1\end{aligned}}  \def\b#1{\mathbf{#1}} \def\t#1{\hbox{#1}}
\def\a#1{\begin{align*}#1\end{align*}} \def\an#1{\begin{align}#1\end{align}}

\newtheorem{WG}{Weak Galerkin Algorithm}

\title{Weak Galerkin methods for elliptic interface problems on curved polygonal partitions}

\author{
Dan Li\thanks {Jiangsu Key Laboratory for NSLSCS, School of Mathematical Sciences,  Nanjing Normal University, Nanjing 210023, China (danlimath@163.com). The research of Dan Li was supported by   Jiangsu Funding Program for Excellent
Postdoctoral Talent (No. 2023ZB271) and National Natural Science Foundation of China (Grant No. 12071227 and No. 12371369).}
\and
Chunmei Wang \thanks{Department of Mathematics, University of Florida, Gainesville, FL 32611 (chunmei.wang@ufl.edu).  The research of Chunmei Wang was partially supported by National Science Foundation Grants DMS-2136380 and DMS-2206332.}
\and
Shangyou Zhang \thanks{Department of Mathematical Sciences, University of Delaware, Newark, DE 19716 (szhang@udel.edu).  }
}

\begin{document}

\maketitle

\begin{abstract}
This paper presents a new weak Galerkin (WG) method for elliptic interface problems on general curved polygonal partitions.  The method's key innovation lies in its ability to transform the complex interface jump condition into a more manageable Dirichlet boundary condition, simplifying the theoretical analysis significantly. The numerical scheme is designed by using locally constructed weak gradient on the curved polygonal partitions. We establish error estimates of optimal order for the numerical approximation in both discrete $H^1$ and $L^2$ norms. Additionally, we present various numerical results that serve to illustrate the robust numerical performance of the proposed WG interface method.

\end{abstract}



\section{Introduction}\label{Section:Introduction}
This paper focuses on the latest advancements in the Weak Galerkin finite element method for solving elliptic interface problems on curved polygonal partitions. To simplify our analysis, we concentrate on a model equation seeking an unknown function $u$ that satisfies:
\begin{eqnarray}
-\nabla\cdot({\color{black}{a}}\nabla u)&=&f,\quad~~ \mbox{in}\;\O,\label{model-1}\\
        u&=&g,\quad~~\mbox{on}\;\pa\O\setminus\Gamma,\label{model-2}\\
{[[u]]}_{\Gamma}=u|_{\O_1}-u|_{\O_2}&=&g_D,\quad \mbox{on}\;\Gamma,\label{model-3}\\
 {[[{\color{black}{a}}\nabla u\cdot \bn]]}_{\Gamma}
 ={\color{black}{a_1}}\nabla u|_{\O_1}\cdot \bn_1+{\color{black}{a_2}}\nabla u|_{\O_2}\cdot \bn_2&=&g_N,\quad \mbox{on}\;\Gamma,\label{model-4}
\end{eqnarray}
where $\O\subset\mathbb R^2$, $\O=\O_1\cup\O_2$, $\Gamma=\O_1\cap\O_2$, ${\color{black}{a_1}}={\color{black}{a}}|_{\O_1}$, ${\color{black}{a_2}}={\color{black}{a}}|_{\O_2}$, $\bn_1$ and $\bn_2$ represent the unit outward normal vectors to $\O_1\cap\Gamma$ and $\O_2\cap\Gamma$, respectively. Assume the diffusion tensor ${\color{black}{a}}$ is symmetric and uniformly positive definite matrix in $\O$.

A weak formulation of the model equation \eqref{model-1}-\eqref{model-4} is as follows: Find $u\in H^1(\O)$,  such that $u=g$ on
$\pa\O\setminus\Gamma$, ${[[u]]}_{\Gamma}=g_D$ on $\Gamma$, satisfies  
\begin{equation}\label{weak-formula}
({\color{black}{a}}\nabla u,\nabla v)=(f,v)+\langle g_N,v\rangle_\Gamma,\quad \forall v\in H_0^1(\O),
\end{equation}
where $H_0^1(\O)=\{v\in H^1(\Omega), v=0\ \text{on}\ \partial\Omega\}$.

Elliptic interface problems find applications in various fields of engineering and science, including biological systems \cite{KLL2009}, material science \cite{HLOZ1997}, fluid dynamics \cite{AL2009}, computational electromagnetic \cite{JH2003,CCCGW2011}. The presence of a discontinuous diffusion tensor in these problems results in solutions that exhibit discontinuities and/or lack smoothness across the interface. This low regularity of the solution presents a significant challenge in the development of high-order numerical methods. To address the mesh constraints associated with interface problems effectively, researchers have proposed several numerical techniques. These methods include interface-fitted mesh approaches, which involve modifying finite element meshes near the interface, and unfitted mesh methods, which alter the finite element discretization around the interface.

Unfitted mesh methods have garnered significant attention for their ability to utilize finite element meshes independently of the interface. They offer two primary strategies for handling interface elements. One approach involves adapting the finite element basis near the interface to construct a finite element space that satisfies the interface jump condition. This strategy encompasses methods like the immersed interface method \cite{ZL1998,JWCL2022,MZ2019,CCGL2022}, ghost fluid methods \cite{LS2003}, multiscale finite element methods \cite{CGH2010}, hybridizable discontinuous Galerkin methods \cite{DWX2017,HCWX2020}. 
Alternatively, another approach employs penalty terms across the interface to enforce the interface jump condition. This category includes methods like extended finite element methods  \cite{XXW2020,CCW2022}, unfitted finite element methods \cite{HH2002}, cut finite element methods \cite{BCHLM2015}, high-order hybridizable discontinuous Galerkin method \cite{HNPK2013}. Despite the successes achieved by unfitted mesh methods, several challenges remain. In particular, accurately capturing interface information for problems with highly complex interface geometries poses difficulties. Additionally, establishing rigorous convergence analyses for high-order numerical methods remains a challenging task.

As an alternative approach, several interface-fitted mesh methods have been developed to tackle elliptic interface problems. These methods aim to accommodate poorly generated meshes and situations with hanging nodes, particularly in the context of complex interfaces. Some notable methods include the discontinuous Galerkin method \cite{HNPK2013,LW2022,WGC2022}, the matched interface and boundary method method \cite{YW2007,ZZFW2006}, virtual element method \cite{CWW2017} and weak Galerkin methods \cite{MWWYZ2013,MWYZ2016, wz2023, CWW2022}. The WG methods, first introduced in  \cite{ellip_JCAM2013} and further developed in \cite{lww2023,lww20232,lw2023,CWW2023, CWW20222, cwyz,cw2020, cw2018} represent a novel class of numerical techniques for solving partial differential equations. Their primary innovation lies in the introduction of weak differential operators and weak functions, which grant WG methods several advantages. Notably, constructing high-order WG approximating functions becomes straightforward, as the continuity requirements for numerical approximations are relaxed. Furthermore, this relaxation of continuity requirements endows WG methods with high flexibility, particularly on general polygonal meshes with straight edges. However, when employing straight-edge elements to discretize curved regions, high-order numerical methods may suffer from reduced accuracy. To mitigate geometric errors arising from the transition between straight-edge and curved-edge regions, one approach is to directly utilize curved-edge elements for discretizing curved geometries \cite{HLY2020,VRC2019}.

The objective of this paper is to introduce a novel Weak Galerkin (WG) method designed for solving elliptic interface problems on general curved polygonal partitions. The new WG method is designed by using locally constructed weak gradient operator on the curved elements. Moreover, the error estimates of optimal order are established for the high order numerical approximation in discrete $H^1$ norm and usual $L^2$ norms.  What sets our approach apart from existing results on standard weak Galerkin methods is that it does not necessitate locally denser meshes near the interface. As a result, our proposed method not only significantly reduces the storage space and computational complexity but also offers greater flexibility in addressing complex interface geometries.

The remainder of the paper is structured as follows:
In Section \ref{Section:definition}, we provide a concise overview of the computation of the weak gradient operator and its discrete counterpart.
Section \ref{Section:WG-scheme} outlines the application of the Weak Galerkin method to solve the model problem described by equations \eqref{model-1} through \eqref{model-4}, based on the weak formulation presented in equation \eqref{weak-formula}.
Section \ref{Section:error-equation} derives an error equation relevant to the Weak Galerkin algorithm.
Section \ref{Section:ErrorEstimate} is focused on establishing error estimates of optimal order for the corresponding numerical approximations, considering both discrete $H^1$ and conventional $L^2$ norms.
Finally, in Section \ref{Section:EE}, we illustrate the practical application of the theoretical results through several numerical examples.

This paper will adhere to the standard notations for Sobolev spaces and norms, as detailed in \cite{GT1983}.
Let $D$ be an open, bounded domain with a Lipschitz continuous boundary denoted as $\partial D$ in $\mathbb{R}^2$. We employ the symbols $(\cdot, \cdot){s, D}$, $|\cdot|{s, D}$, and $|\cdot|{s, D}$ to represent the inner product, seminorm, and norm within the Sobolev space $H^s(D)$ where $s\geq 0$ is an integer.  In the case of $s=0$, we denote the inner product and norm as $(\cdot, \cdot){D}$ and $|\cdot|_{D}$, respectively.  When $D=\O$,  we omit the subscript $D$ in the corresponding inner product and norm notation. For the sake of simplicity, we use the notation "$A\lesssim B$" to express the inequality "$A\leq CB$," where $C$ represents an arbitrary positive constant that remains independent of mesh size or functions involved in the inequalities.

\section{Weak Gradient and Discrete Weak Gradient}\label{Section:definition}

The objective of this section is to provide a review of the definitions for the weak gradient operator and its discrete counterpart, as outlined in \cite{ellip_JCAM2013} and \cite{WY-ellip_MC2014}. To facilitate this review, consider a polygonal domain $T$ with a boundary $\partial T$ that is Lipschitz continuous. 

In this context, a weak function defined on $T$ is represented as $v={v_0,v_b}$, where $v_0\in L^2(T)$ and $v_b\in L^2(\partial T)$. The first component, $v_0$, and the second component, $v_b$, correspond to the values of $v$ within the interior of $T$ and on the boundary of $T$, respectively. It's worth noting that $v_b$ may not necessarily be the trace of $v_0$ on $\partial T$.

Let $\mathcal{W}(T)$ denote the space encompassing all such weak functions on $T$:
$$\W(T)=\{v=\{v_0,v_b\},~v_0\in L^2(T),~v_b\in L^2(\pa T)\}.$$

\begin{definition}(Weak gradient)
For any $v\in\W(T)$, the weak gradient of $v$, denoted as $\nabla_{w}v$, is defined as a linear functional in the dual space of $[H^1(T)]^2$ such that
\begin{eqnarray}\label{weak-grad}
(\nabla_{w}v,\pmb{\psi})_T=-(v_0,\nabla\cdot\pmb{\psi})_T+\langle v_b,\pmb{\psi}\cdot\bn\rangle_{\pa T},~~~\forall\pmb{\psi}\in[H^1(T)]^2,
\end{eqnarray}
where $\bn$ denotes the unit outward normal vector to $\pa T$.
\end{definition}

For any non-negative integer $r$, we denote by $P_r(T)$ the set of polynomials defined on the polygonal domain $T$ with a degree not exceeding $r$.

\begin{definition}(Discrete weak gradient)
A discrete form of $\nabla_{w}v$ for $v\in\W(T)$, denoted by $\nabla_{w,r,T}v$, is defined as a unique polynomial vector in $[P_r(T)]^2$ satisfying
\begin{eqnarray}\label{discre-weak-grad}
(\nabla_{w,r,T}v,\pmb{\psi})_T=-(v_0,\nabla\cdot\pmb{\psi})_T+\langle v_b,\pmb{\psi}\cdot\bn\rangle_{\pa T},~~~\forall\pmb{\psi}\in[P_r(T)]^2.
\end{eqnarray}
\end{definition}

\section{Weak Galerkin Scheme}\label{Section:WG-scheme}

In this section, we present the Weak Galerkin scheme for the model problems described by equations \eqref{model-1} through \eqref{model-4}. To facilitate this, consider $\mathcal{T}_h$, a curved polygonal partition of $\Omega$, which conforms to the shape regularity criteria outlined in \cite{WWL-2022-cur}. For the sake of simplicity, Figure \ref{Pictu-model} displays a curved triangular partition of a square domain, denoted as $\Omega$. It's worth noting that when the interface $\Gamma$ is curved, $\mathcal{T}_h$ fits seamlessly along the interface.

We denote $\mathcal{E}_h$ as the set encompassing all edges within $\mathcal{T}_h$, and $\mathcal{E}_h^0$ as the set of all interior edges, excluding those along $\partial\Omega$. Additionally, $\Gamma_h$ is defined as the set of interface edges within $\mathcal{E}_h$.   $h_T$ represents the diameter of an element $T\in\mathcal{T}_h$, and $h$ is the mesh size, defined as the maximum of $h_T$ over all $T\in\mathcal{T}_h$. Lastly, $|e|$ denotes the length of an edge $e\in\mathcal{E}_h$.

 \begin{figure}[!htbp]
\centering
\includegraphics[height=0.40\textwidth,width=0.40\textwidth]{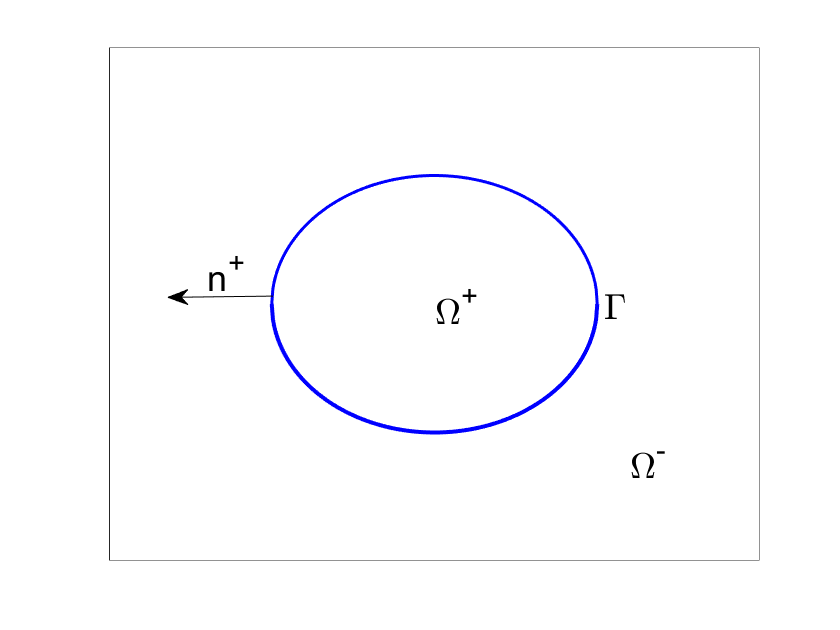}
\includegraphics[height=0.40\textwidth,width=0.40\textwidth]{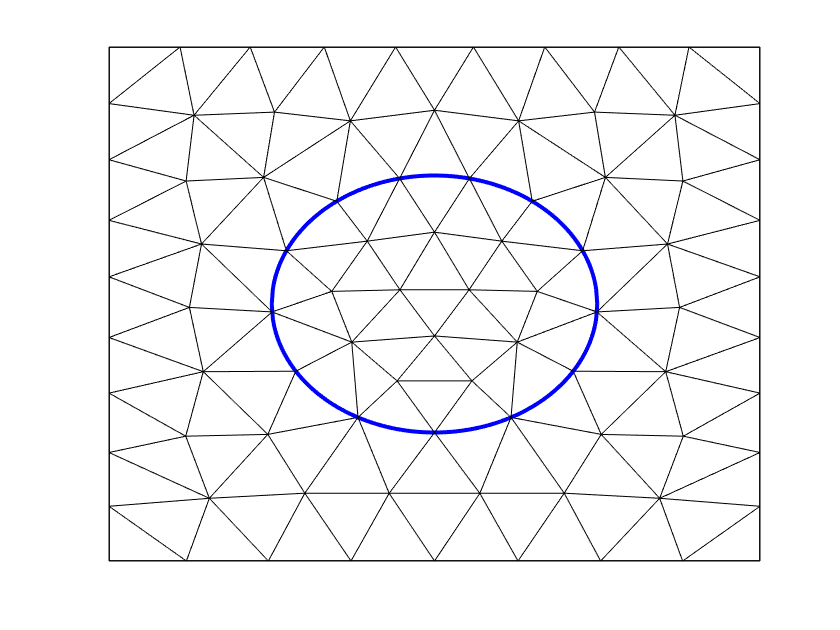}
\caption{The geometry of domain $\O=\O_1\cup\O_2\cup\Gamma$ with smooth interface $\Gamma$ (Left) and a fitted partition (Right).}
\label{Pictu-model}
\end{figure}

Let $e$ be the curved edge of the curved element $T$. Suppose that the parametric representation for edge $e$ is given by:
\[\bx={F}_e(\hat{t}),\qquad\hat{t}\in\hat{e}=[0,|e|],\]
where $\bx=(x,y)\in e$, ${F}_e(\hat{t})=(\phi(\hat{t}),\psi(\hat{t}))$, $\phi(\hat{t})\in C^n(\hat{e})$, $\psi(\hat{t})\in C^n(\hat{e})$ for some $n\geq1$.  In this context, ${F}_e:=(\phi,\psi)$ represents the mapping that transforms the curved edge $e$ to its corresponding straight edge $\hat{e}$, and we assume that this mapping ${F}_e$ is globally invertible on the reference edge $\hat{e}.$ Then, ${F}_e$ and its inverse mapping $\widehat{F}_e:={F}_e^{-1}$ can be extended to encompass the entire "pyramid" region, as discussed in  \cite{WWL-2022-cur}.

For any function $\widehat{w}\in L^2(\hat{e})$, we can use the mapping $\widehat{F}_e$ to obtain a function $w\in L^2(e)$ as follows:
\begin{equation}\label{EQ:March08.001}
w(\bx):=\hw(\widehat{F}_e(\bx)),\qquad\bx\in e.
\end{equation}
Similarly, any function $w\in L^2(e)$ can be transformed into a function $\hw\in L^2(\he)$ given by
\begin{equation}\label{EQ:March08.002}
\hw(\hat{t}):=w({F}_e(\hat{t})),\qquad\hat{t}\in\he.
\end{equation}
Consequently, we have the relationships:
$$
w=\hw\circ\widehat{F}_e,\quad\hw=w\circ{F}_e.
$$

Let $\ell \ge 0$ be any non-negative integer. We denote by ${P}_{\ell}(\he)$ the set of polynomials defined on the straight edge $\he$ with a degree no greater than $\ell$. By utilizing the mapping $\widehat{F}_e := F_e^{-1}$, we can transform the set of polynomials ${P}_{\ell}(\he)$ into a space of functions defined on the curved edge $e$. This transformed space is denoted as follows:
$$
V_b(e,\ell)=\{w=\widehat{w}\circ\widehat{F}_e:\ \ \widehat{w}\in{P}_{\ell}(\he)\}.
$$
Moreover, when the edge $e$ is a straight edge, we make the assumption that the mapping $F_e$ is an affine transformation. Consequently, the inverse mapping $\widehat{F}_e$ is also an affine transformation. In this special case, it follows that 
$$V_b(e,\ell)=P_{\ell}(e).$$

Let $k\geq1$ be any given integer.  When the  edge $e$ is on the interface $\Gamma_h$,  $v_b$ is differently valued as seen from the left side $e_L$ and from the right   side $e_R$; otherwise, $v_b$ is single  valued on the edge  $e\in\E_h^0\setminus\Gamma_h$. Denote by $V_h$ the finite element space associated with ${\cal T}_h$ as follows
 \an{\label{V-h}  \ad{
V_h= \{v=\{v_0,v_b\} : &  \ v_0\in P_k(T),~T\in{\cal T}_h, 
            \ v_b|_e\in V_b(e,k-1), e\in\E_h,  \\&
          v_b|_{e_L}\neq   v_b|_{e_R},  e\in\Gamma_h  
     \}. } }

Denote by $V_h^0$ a subspace of $V_h$ with homogeneous boundary value for $v_b$; i.e.,
$$V_h^0=\{v\in V_h,~v_b|_{e}=0,~~e\subset\pa\O\}.$$

For simplicity of notation, denote by $\nabla_wv$ the discrete weak gradient $\nabla_{w,r,T}v$ defined by \eqref{discre-weak-grad} on each element $T$ with $r=k-1$; i.e.,
\an{\label{w-g}
(\nabla_wv)|_T=\nabla_{w,r,T}(v|_T),~~~~~v\in V_h. }

For each edge $e\in\E_h$, denote by $Q_b$ the projection operator mapping from $L^2(e)$ to $V_b(e,k-1)$ given by
$$
Q_bw\circ F_e:=\hQ_b(w\circ F_e),\qquad w\in L^2(e),
$$
where $\hQ_b$ is the weighted $L^2$ projection operator onto ${P}_{k-1}(\he)$ with the corresponding Jacobian as the weight function. Note that when $e$ is a straight edge, the operator $Q_b$ represents the standard $L^2$ projection operator onto $P_{k-1}(e)$.

For any edge $e\in\Gamma_h$ shared by two adjacent elements $T_1\subset \O_1$ and $T_2\subset \O_2$, we denote by ${[[v_b]]}_{\Gamma_h}$ the jump of $v_b$ on $e\in\Gamma_h$; i.e.,
$${[[v_b]]}_{\Gamma_h}=v_b|_{\pa T_1\cap\Gamma_h}-v_b|_{\pa T_2\cap\Gamma_h}.$$

For any $v,w\in V_h$, let us introduce the following bilinear forms:
\begin{eqnarray*}
s(v,w)&=&\rho\sum_{T\in{\cal T}_h}h_T^{-1}\langle Q_bv_0-v_b,Q_bw_0-w_b\rangle_{\pa T},\\
a(v,w)&=&\sum_{T\in{\cal T}_h}({\color{black}{a}}\nabla_wv,\nabla_ww)_T+s(v,w),
\end{eqnarray*}
where $\rho>0$ is the stabilization parameter.

\begin{WG}
A weak Galerkin numerical scheme for the weak formulation \eqref{weak-formula} of  the model problem \eqref{model-1}-\eqref{model-4}  can be obtained by seeking $u_h=\{u_0,u_b\}\in V_h$ such that $u_b=Q_bg$ on $\pa\O$, ${[[u_b]]}_{\Gamma_h}=Q_bg_D$  satisfying
\begin{equation}\label{wg}
\begin{split}
a(u_h,v)=&(f,v_0)+\sum_{e\in\Gamma_h\cap\O_1}\l g_N,v_b\r_e+\sum_{e\in\Gamma_h\cap\O_2}\l g_N,v_b\r_e,\quad \forall v\in V_h^0.
\end{split}
\end{equation}
\end{WG}


\section{Error Equation}\label{Section:error-equation}

This section aims to derive an error equation for the weak Galerkin scheme \eqref{wg}. For simplicity of analysis, we assume that the coefficient tensor ${\color{black}{a}}$ in the model problem \eqref{model-1}-\eqref{model-4} is piecewise constant with respect to the finite element partition $\T_h$. The following analysis can be generalized to piecewise smooth   tensor ${\color{black}{a}}$ without technical difficulty.

Let $u$ and $u_h\in V_h$ be the exact solution of the model problem \eqref{model-1}-\eqref{model-4} and the numerical solution of the WG scheme \eqref{wg}, respectively. On each element $T\in\T_h$, denote by $Q_0$ the usual $L^2$ projection operator onto $P_k(T)$.  Recall that $Q_bu$ takes different values as seen from the left side and right side of the edge $e\subset\Gamma_h$ and takes a single value on the edge $e \subset \E_h^0\setminus\Gamma_h$. We further define a projection $Q_hu$ onto $V_h$ such that
$$Q_hu=\{Q_0u,Q_bu\}.$$
Denote by $\bbQ_h$ the $L^2$ projection operator onto $[P_{k-1}(T)]^2.$

Let the error function $e_h$  be defined by
$$e_h=Q_hu-u_h=\{e_0,e_b\}=\{Q_0u-u_0,Q_bu-u_b\}.$$

\begin{lemma}\cite{WWL-2022-cur}\label{commu-property}
For any $\pmb{\psi}\in[P_{k-1}(T)]^2$, there holds
\begin{equation*}
\begin{split}
(\nabla_wQ_hu,\pmb{\psi})_T&=(\nabla u,\pmb{\psi})_T+\l Q_bu-u,\pmb{\psi}\cdot\bn\r_\pT.
\end{split}
\end{equation*}
Note that $\l Q_bu-u,\pmb{\psi}\cdot\bn\r_\pT\neq0$ when the boundary $\pa T$   consists of at least one curved edge.
\end{lemma}

\begin{lemma}\label{error-equ-0}
For any $v\in V_h^0$, the error function $e_h$ satisfies the following equation
\begin{equation*}
a(e_h,v)=s(Q_hu,v)+\ell_1(u,v)+\ell_2(u,v),
\end{equation*}
where $\ell_1(u,v)$ and $\ell_2(u,v)$ are given by
\begin{eqnarray*}
\ell_1(u,v)&=&\sum_{T\in\T_h}\langle({\color{black}{a}}\nabla u-{\color{black}{a}}\bbQ_h\nabla u)\cdot\bn,v_0-v_b\rangle_\pT,\\
\ell_2(u,v)&=&\sum_{T\in\T_h}\l Q_bu-u,{\color{black}{a}}\nabla_wv\cdot\bn\r_\pT.
\end{eqnarray*}
Note that the last term $\ell_2(u,v)=0$ when the boundary $\pa T$ are   straight edges.
\end{lemma}

\begin{proof}
By testing the model equation \eqref{model-1} against $v_0$ and then using the usual integration by parts, there holds
\begin{equation}\label{error-equ-1}
\begin{split}
&\sum_{T\in\T_h}(-\nabla\cdot({\color{black}{a}}\nabla u),v_0)_T\\
=&\sum_{T\in\T_h}({\color{black}{a}}\nabla u,\nabla v_0)_T-\l{\color{black}{a}}\nabla u\cdot\bn,v_0\r_\pT\\
=&\sum_{T\in\T_h}({\color{black}{a}}\nabla u,\nabla v_0)_T-\l{\color{black}{a}}\nabla u\cdot\bn,v_0-v_b\r_\pT-
  \sum_{e\in\Gamma_h}\l{\color{black}{a}}\nabla u\cdot\bn,v_b\r_e\\
=&\sum_{T\in\T_h}({\color{black}{a}}\nabla u,\nabla v_0)_T-\l{\color{black}{a}}\nabla u\cdot\bn,v_0-v_b\r_\pT-
\sum_{e\in\Gamma_h\cap\O_1}\l{[[{\color{black}{a}}\nabla u\cdot\bn]]},v_b\r_e\\
&-\sum_{e\in\Gamma_h\cap\O_2}\l{[[{\color{black}{a}}\nabla u\cdot\bn]]},v_b\r_e\\
=&\sum_{T\in\T_h}({\color{black}{a}}\nabla u,\nabla v_0)_T-\l{\color{black}{a}}\nabla u\cdot\bn,v_0-v_b\r_\pT-
  \sum_{e\in\Gamma_h\cap\O_1}\l g_N,v_b\r_e\\
&-\sum_{e\in\Gamma_h\cap\O_2}\l g_N,v_b\r_e,
\end{split}
\end{equation}
where we also used the boundary condition \eqref{model-4} and the fact that $v_b$ is single valued on   $e\in\E_h\setminus\Gamma_h$.

For the first term on the last line of \eqref{error-equ-1}, using the definition of $\bbQ_h$, the usual integration by parts and \eqref{discre-weak-grad}  yields
\begin{equation}\label{error-equ-2}
\begin{split}
({\color{black}{a}}\nabla u,\nabla v_0)_T
=&({\color{black}{a}}\bbQ_h\nabla u,\nabla v_0)_T\\
=&-(\nabla\cdot({\color{black}{a}}\bbQ_h\nabla u),v_0)_T+\langle{\color{black}{a}}\bbQ_h\nabla u\cdot\bn,v_0\rangle_{\pa T}\\
=&({\color{black}{a}}\bbQ_h\nabla u,\nabla_wv)_T-\langle v_b,{\color{black}{a}}\bbQ_h\nabla u\cdot\bn\rangle_{\pa T}+\langle{\color{black}{a}}\bbQ_h\nabla u\cdot\bn,v_0\rangle_{\pa T}\\
=&({\color{black}{a}}\bbQ_h\nabla u,\nabla_wv)_T+\langle{\color{black}{a}}\bbQ_h\nabla u\cdot\bn,v_0-v_b\rangle_{\pa T}.
\end{split}
\end{equation}
Substituting \eqref{error-equ-2} into \eqref{error-equ-1} and then using the definition of $\bbQ_h$, Lemma \ref{commu-property} with $\pmb{\psi}={\color{black}{a}}\nabla_wv$ give
\begin{equation}\label{error-equ-3}
\begin{split}
&\sum_{T\in\T_h}(-\nabla\cdot({\color{black}{a}}\nabla u),v_0)_T\\
=&\sum_{T\in\T_h}({\color{black}{a}}\bbQ_h\nabla u,\nabla_wv)_T
-\langle v_0-v_b,{\color{black}{a}}(\nabla u-\bbQ_h\nabla u)\cdot\bn\rangle_{\pa T}\\
&-\sum_{e\in\Gamma_h\cap\O_1}\l g_N,v_b\r_e-\sum_{e\in\Gamma_h\cap\O_2}\l g_N,v_b\r_e\\
=&\sum_{T\in\T_h}({\color{black}{a}}\nabla u,\nabla_wv)_T-\ell_1(u,v)-\sum_{e\in\Gamma_h\cap\O_1}\l g_N,v_b\r_e-\sum_{e\in\Gamma_h\cap\O_2}\l g_N,v_b\r_e\\
=&\sum_{T\in\T_h}(\nabla_wQ_hu,{\color{black}{a}}\nabla_wv)_T-\langle Q_bu-u,{\color{black}{a}}\nabla_wv\cdot\bn\rangle_{\pa T}-\ell_1(u,v)\\
 &-\sum_{e\in\Gamma_h\cap\O_1}\l g_N,v_b\r_e-\sum_{e\in\Gamma_h\cap\O_2}\l g_N,v_b\r_e\\
=&\sum_{T\in\T_h}(\nabla_wQ_hu,{\color{black}{a}}\nabla_wv)_T-\ell_2(u,v)-\ell_1(u,v)
   -\sum_{e\in\Gamma_h\cap\O_1}\l g_N,v_b\r_e\\
   &-\sum_{e\in\Gamma_h\cap\O_2}\l g_N,v_b\r_e.
\end{split}
\end{equation}
Using \eqref{model-1}, \eqref{wg} and $e_h=Q_hu-u_h$, one arrives at
$$a(e_h,v)=s(Q_hu,v)+\ell_1(u,v)+\ell_2(u,v),$$
which completes the proof of the lemma.
\end{proof}

\section{Technical Results}\label{Section:TechnicalResults}

This section is devoted to presenting some technical results. To this end, let $\T_h$ be a curved shape regular partition as described in \cite{WWL-2022-cur}. For any $T\in\T_h$ and $\phi\in H^1(T)$, the following trace inequality holds true \cite{WWL-2022-cur}:
\begin{equation}\label{Trace-inequality}
\|\phi\|_e^2\lesssim h_T^{-1}\|\phi\|_T^2+h_T\|\nabla\phi\|_{T}^2.
\end{equation}
If $\phi$ is a polynomial on any $T\in\T_h$, using the inverse inequality, there holds \cite{WWL-2022-cur}
\begin{equation}\label{Trace-inver-inequality}
\|\phi\|_e^2\lesssim h_T^{-1}\|\phi\|_T^2.
\end{equation}

\begin{lemma}\label{LEMMA:error-estimates}
Let $\T_h$ be a curved finite element partition of $\O$ that is shape regular as described in \cite{WWL-2022-cur}. For any $\phi\in H^{k+1}(\O)$, there holds \cite{WWL-2022-cur}
\begin{eqnarray}
\sum_{T\in\T_h}h_T^{2s}\|Q_0\phi-\phi\|_{s,T}^2&\lesssim&h^{2k+2}\|\phi\|_{k+1}^2,~~~~0\leq s\leq2,\label{error-estimates-1}\\
\sum_{T\in\T_h}h_T^{2s}\|\nabla \phi-\bbQ_h\nabla \phi\|_{s,T}^2&\lesssim&h^{2k}\|\phi\|_{k+1}^2,~~~~0\leq s\leq2,\label{error-estimates-2}\\
\sum_{T\in\T_h}\|Q_b\phi-\phi\|_\pT^2&\lesssim&h^{2k-1}\|\phi\|_{k}^2.\label{error-estimates-3}
\end{eqnarray}
\end{lemma}

\begin{lemma}
For any $v\in V_h$, $\phi\in H^1(T)$ and $\bq\in[P_{k-1}(T)]^2$, there holds \cite{WWL-2022-cur}
\begin{eqnarray}
h_T^{-1}\|v_0-v_b\|_\pT^2&\lesssim&\|\nabla v_0\|_T^2+h_T^{-1}\|Q_bv_0-v_b\|_\pT^2,\label{happy.000}\\
\|\nabla v_0\|_T^2&\lesssim&\|\nabla_w v\|_T^2+h_T^{-1}\|Q_bv_0-v_b\|_\pT^2,\label{happy.001}
\end{eqnarray}
\begin{equation}\label{happy.002}
|\langle\phi-Q_b\phi,\bq\cdot\bn\rangle_e|\lesssim\left\{
\begin{array}{ll}
h_e^{1/2}\|\phi-Q_b\phi\|_\pT \,\|\bq\|_{T},\qquad\qquad\qquad \mbox{for $k\ge 1$,}\\
h_e^{3/2}\|\phi-Q_b\phi\|_\pT \,(\|\bq\|_{T}+\|\nabla \bq\|_T),\quad\mbox{for $k\ge 2$}.
\end{array}
\right.
\end{equation}
\end{lemma}

For any $v\in V_h$, the weak Galerkin scheme \eqref{wg} induces a semi norm given by
\begin{equation}\label{semi-norm}
\3barv\3bar^2=a(v,v).
\end{equation}

\begin{lemma}\label{LEMMA:triplebarnorm}
For any $v\in V_h^0$, the semi norm defined in \eqref{semi-norm} is a norm.
\end{lemma}
\begin{proof}
The proof is similar to the proof of Lemma 5.1 in \cite{WWL-2022-cur}.
\end{proof}

%
%
%
%
%


\begin{lemma}\label{error-equation-estimate}
For any $u\in H^{k+1}(\O_i)$ for $i=1,2$ and $v\in V_h$, there holds
\begin{eqnarray}
|s(Q_hu,v)|&\lesssim&h^k(\|u\|_{k+1,\O_1}+\|u\|_{k+1,\O_2})\3bar v\3bar,\label{mmm1}\\
\left|\ell_1(u,v)\right|&\lesssim&h^k(\|u\|_{k+1,\O_1}+\|u\|_{k+1,\O_2})\3barv\3bar,\label{mmm2}\\
\left|\ell_2(u,v)\right|&\lesssim&h^k(\|u\|_{k+1,\O_1}+\|u\|_{k+1,\O_2})\3barv\3bar,\label{mmm3}
\end{eqnarray}
where $\ell_1(u,v)$ and $\ell_2(u,v)$ are given by Lemma \ref{error-equ-0}.
\end{lemma}

\begin{proof} To derive the first inequality (\ref{mmm1}), it follows from the Cauchy-Schwarz inequality, the property of $Q_b$, \eqref{Trace-inequality} and \eqref{error-estimates-1} that
\begin{eqnarray*}
|s(Q_hu, v)|
&=&|\rho\sum_{T\in\T_h}h_T^{-1}\langle Q_b(Q_0u)-Q_bu,Q_bv_0-v_b\rangle_\pT|\\
&\lesssim&\Big(\rho\sum_{T\in\T_h}h_T^{-1}\|Q_b(Q_0u-u)\|^2_{\pT}\Big)^{\frac12}
      \Big(\rho\sum_{T\in\T_h}h_T^{-1}\|Q_bv_0-v_b\|^2_{\pT}\Big)^{\frac12}\\
&\lesssim&\Big(\sum_{T\in\T_h}h_T^{-1}\|Q_0u-u\|^2_{\pT}\Big)^{\frac12}\3barv\3bar\\
&\lesssim&\Big(\sum_{T\in\T_h}h_T^{-1}h_T^{-1}\|Q_0u-u\|^2_{T}+h_T^{-1}h_T|Q_0u-u|^2_{1,T}\Big)^{\frac12}\3barv\3bar\\
&\lesssim&h^k(\|u\|_{k+1,\O_1}+\|u\|_{k+1,\O_2})\3bar v\3bar.
\end{eqnarray*}

To analyze the second inequality (\ref{mmm2}), using the Cauchy-Schwarz inequality \eqref{Trace-inequality}, \eqref{error-estimates-2}, \eqref{error-estimates-3}, \eqref{happy.001} and \eqref{semi-norm}, there holds
\begin{equation}\label{EQ:July7-2016:001}
\begin{split}
&|\ell_1(u,v)|\\
=&|\sum_{T\in\T_h}\langle{\color{black}{a}}\nabla u\cdot\bn-{\color{black}{a}}\bbQ_h\nabla u\cdot\bn,v_0-v_b\rangle_\pT|\\
\lesssim&\Big(\sum_{T\in\T_h} h_T\|{\color{black}{a}}\nabla u-{\color{black}{a}}\bbQ_h\nabla u\|_{\pT}^2\Big)^{\frac12}
    \Big(\sum_{T\in\T_h} h_T^{-1}\|v_0-v_b\|_\pT^2\Big)^{\frac12}\\
\lesssim&\Big(\sum_{T\in\T_h} h_Th_T^{-1}\|{\color{black}{a}}\nabla u-{\color{black}{a}}\bbQ_h\nabla u\|_{T}^2
+h_Th_T|{\color{black}{a}}\nabla u-{\color{black}{a}}\bbQ_h\nabla u|_{1,T}^2\Big)^{\frac12}\\
&\cdot\Big(\sum_{T\in\T_h} h_T^{-1}\|Q_bv_0-v_b\|_\pT^2+h_T^{-1}\|v_0-Q_bv_0\|_\pT^2\Big)^{\frac12}\\
\lesssim&h^k(\|u\|_{k+1,\O_1}+\|u\|_{k+1,\O_2})\Big(\3bar v\3bar^2+\sum_{T\in\T_h}h_T^{-1}h_T\|\nabla v_0\|_T^2\Big)^{\frac12}\\
\lesssim&h^k(\|u\|_{k+1,\O_1}+\|u\|_{k+1,\O_2})
\Big(\3bar v\3bar^2+\sum_{T\in\T_h}\|\nabla_w v\|_T^2+h_T^{-1}\|Q_bv_0-v_b\|_{\pT}^2\Big)^{\frac12}\\
\lesssim&h^k(\|u\|_{k+1,\O_1}+\|u\|_{k+1,\O_2})\Big(\3bar v\3bar^2+\3bar v\3bar^2+\3bar v\3bar^2\Big)^{\frac12}\\
\lesssim&h^k(\|u\|_{k+1,\O_1}+\|u\|_{k+1,\O_2})\3bar v\3bar.
\end{split}
\end{equation}

To estimate the last estimate (\ref{mmm3}), from \eqref{happy.002}, the Cauchy-Schwarz inequality, \eqref{error-estimates-3} and \eqref{semi-norm}, there yields
\begin{eqnarray*}
|\ell_2(u,v)|
&=&|\sum_{T\in\T_h}\l Q_bu-u,\color{black}{a}\nabla_wv\cdot\bn\r_\pT|\\
&\lesssim&\sum_{T\in\T_h}h_T^{\frac12}\|Q_bu-u\|_\pT\|\color{black}{a}\nabla_wv\|_T \\
&\lesssim&h^{\frac12}\Big(\sum_{T\in\T_h}\|Q_bu-u\|_\pT^2\Big)^{\frac12}\Big(\sum_{T\in\T_h}\|\color{black}{a}\nabla_wv\|_T^2\Big)^{\frac12}\\
&\lesssim&h^kh^{\frac{2k-1}{2}}(\|u\|_{k+1,\O_1}+\|u\|_{k+1,\O_2})\3barv\3bar\\
&\lesssim&h^k(\|u\|_{k+1,\O_1}+\|u\|_{k+1,\O_2})\3barv\3bar.
\end{eqnarray*}
This completes the proof of the lemma.
\end{proof}

\section{Error Estimates}\label{Section:ErrorEstimate}

The objective of this section is to establish some optimal order error estimates for the numerical approximation.

\begin{theorem}\label{H1-errorestimate}
Let $u$ and $u_h\in V_h$ be the exact solution of the model problem \eqref{model-1}-\eqref{model-4} and the numerical solutions of the WG scheme \eqref{wg}, respectively. Assume that the exact solution $u$ satisfies $u\in H^{k+1}(\O_i)$ for $i=1,2$. Then, the following error estimate holds true
\begin{equation}\label{err1}
\3bare_h\3bar\lesssim h^k(\|u\|_{k+1,\O_1}+\|u\|_{k+1,\O_2}).
\end{equation}
\end{theorem}

\begin{proof}
By taking $v=e_h$ in Lemma \ref{error-equ-0}, one arrives at
\begin{equation*}
a(e_h,e_h)=s(Q_hu,e_h)+\ell_1(u,e_h)+\ell_2(u,e_h).
\end{equation*}
It follows from \eqref{semi-norm} and Lemma \ref{error-equation-estimate} with $v=e_h$ that
$$\3bare_h\3bar^2\lesssim h^k(\|u\|_{k+1,\O_1}+\|u\|_{k+1,\O_2})\3bare_h\3bar.$$

This completes the proof.
\end{proof}

\begin{corollary}\label{Corollary:error-estimate}
Under the assumptions of Theorem \ref{H1-errorestimate}, the following error estimate holds true
\begin{equation}\label{Corollary:error-estimate-1}
\begin{split}
\|\nabla e_0\|\lesssim h^k(\|u\|_{k+1,\O_1}+\|u\|_{k+1,\O_2}).
\end{split}
\end{equation}
\end{corollary}

\begin{proof}
It follows from \eqref{happy.001}, \eqref{semi-norm} and Theorem \ref{H1-errorestimate} that
\begin{eqnarray*}
\|\nabla e_0\|
&\lesssim&\Big(\sum_{T\in\T_h}\|\nabla_w e_h\|_T^2+h_T^{-1}\|Q_be_0-e_b\|_\pT^2\Big)^{\frac12}\\
&\lesssim&\3bare_h\3bar\\
&\lesssim&h^k(\|u\|_{k+1,\O_1}+\|u\|_{k+1,\O_2}).
\end{eqnarray*}
This completes the proof of the corollary.
\end{proof}

\begin{theorem}\label{L2-errorestimate}
Let $u\in H^{k+1}(\O_i)$ for $i=1,2$  be the exact solution of \eqref{model-1}-\eqref{model-4} and $u_h\in V_h$ be the numerical solution of WG scheme \eqref{wg}, respectively. Assume that the dual problem of \eqref{model-1}-\eqref{model-4} satisfies the $H^2$ regular property as described in \cite{ellip_JCAM2013}. Then, the following error estimate holds true
\begin{equation}\label{error-estimate-2-eo}
\begin{split}
\|e_0\|\lesssim h^{k+1}(\|u\|_{k+1,\O_1}+\|u\|_{k+1,\O_2}).
\end{split}
\end{equation}
\end{theorem}

\begin{proof}
The proof is similar to the proof of Theorem 6.4 in \cite{ellippolyreduc2015}.
\end{proof}

To establish the error estimate for $e_b$, we define the following semi-norm
\begin{eqnarray}\label{normed}
\|e_b\|_{\E_h}=\Big(\sum_{T\in\T_h}h_T\|e_b\|_{\pa T}^2\Big)^{1/2}.
\end{eqnarray}

\begin{theorem}\label{L2-errorestimate-eb}
In the assumptions of Theorem \ref{L2-errorestimate}, we have the following error estimate
\begin{equation}
\|e_b\|_{\E_h}\lesssim h^{k+1}(\|u\|_{k+1,\O_1}+\|u\|_{k+1,\O_2}).
\end{equation}
\end{theorem}

\begin{proof}
Using the triangle inequality, the trace inequality \eqref{Trace-inver-inequality}, Theorem  \ref{H1-errorestimate} and  Theorem  \ref{L2-errorestimate}, there holds
\begin{equation*}\label{L2-estimate-edge-2}
\begin{split}
 \|e_b\|_{\E_h}
=&\Big(\sum_{T\in{\cal T}_h}h_T\|e_b\|_{\pa T}^2\Big)^{\frac{1}{2}}\\
\lesssim&\Big(\sum_{T\in{\cal T}_h}h_T\|Q_be_0\|_{\pa T}^2+h_T\|e_b-Q_be_0\|_{\pa T}^2\Big)^{\frac{1}{2}}\\
\lesssim&\Big(\sum_{T\in{\cal T}_h}h_Th_T^{-1}\|e_0\|_{T}^2\Big)^{\frac{1}{2}}
    +\Big(\sum_{T\in{\cal T}_h}h_{T}^{2}\rho_1h_{T}^{-1}\|e_b-Q_be_0\|_{\pa T}^2\Big)^{\frac{1}{2}}\\
\lesssim&\|e_0\|+h\3bare_h\3bar\\
\lesssim&h^{k+1}(\|u\|_{k+1,\O_1}+\|u\|_{k+1,\O_2}).
\end{split}
\end{equation*}
This completes the proof of the theorem.
\end{proof}

\section{Numerical Experiments}\label{Section:EE}

This section presents some numerical experiments to validate the accuracy of the developed convergence theory.

In the first numerical test,  we solve the elliptic interface problem \eqref{weak-formula}:
  Find $u\in H^1(\Omega)$ such that $u = 2-(x^2+y^2)^3$ on $\partial \Omega$ and satisfying 
\an{ (a \nabla u, \nabla v) = (36 (x^2+y^2)^2, v), \quad\forall v\in H^1_0(\Omega),
   \label{e-1} }
where \an{ \label{mu} \ad{
  a(x,y)&=\begin{cases} \mu, & \t{if }\ x^2+y^2 \le 1, \\
                         1,  & \t{if }\ x^2+y^2 > 1, \end{cases} \\
  \Omega&= (-2,2)\times(-2,2).  } }

\begin{figure}[ht]\begin{center}\setlength\unitlength{1.2in} 
    \begin{picture}(4.03,1.1)
\put(0,0.98){$\Gamma$ and $\partial \Omega$:}
\put(1.02,0.98){$G_1$:}
\put(2.03,0.98){$G_2$:}
\put(3.04,0.98){$G_3$:}
 \put(0,-0.35){\includegraphics[width=1.4in]{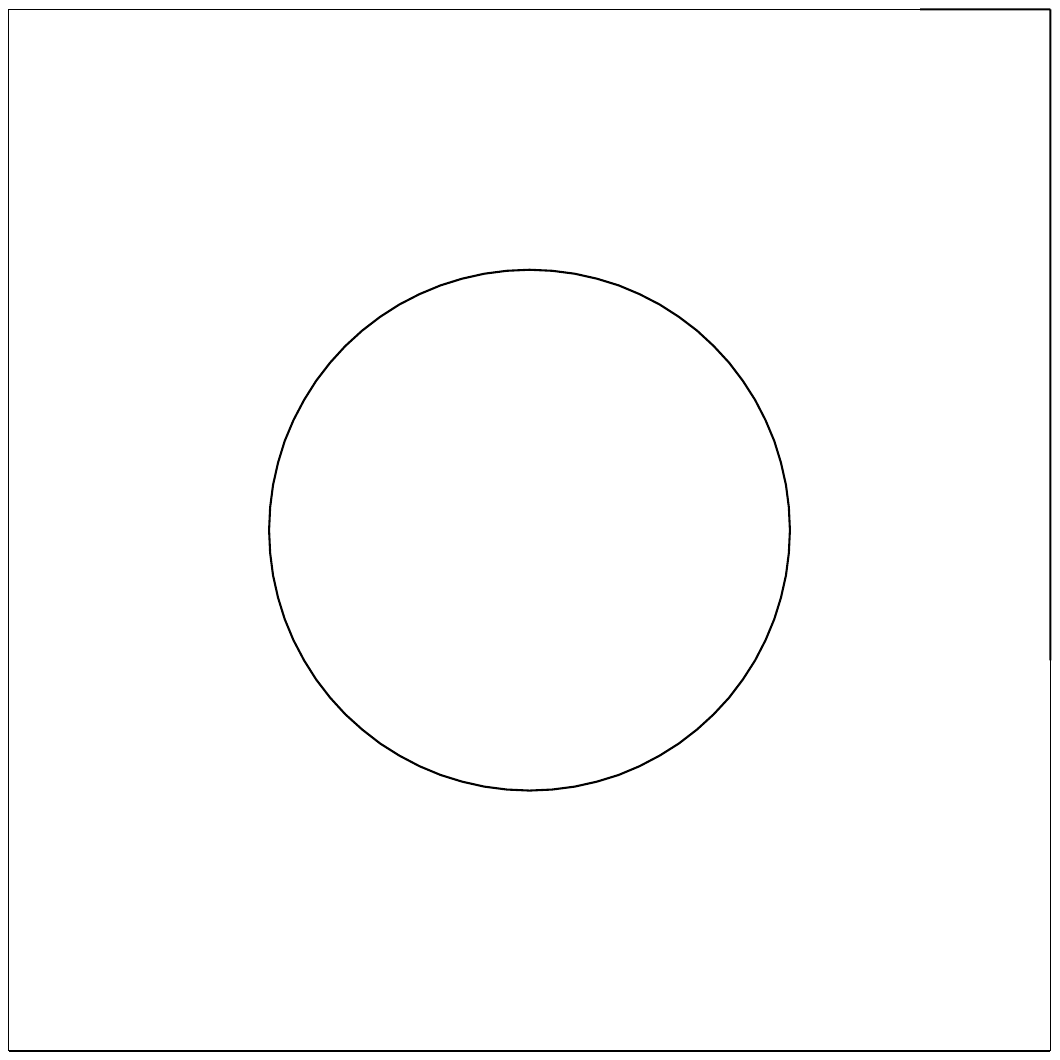}}  
 \put(1.01,-0.35){\includegraphics[width=1.4in]{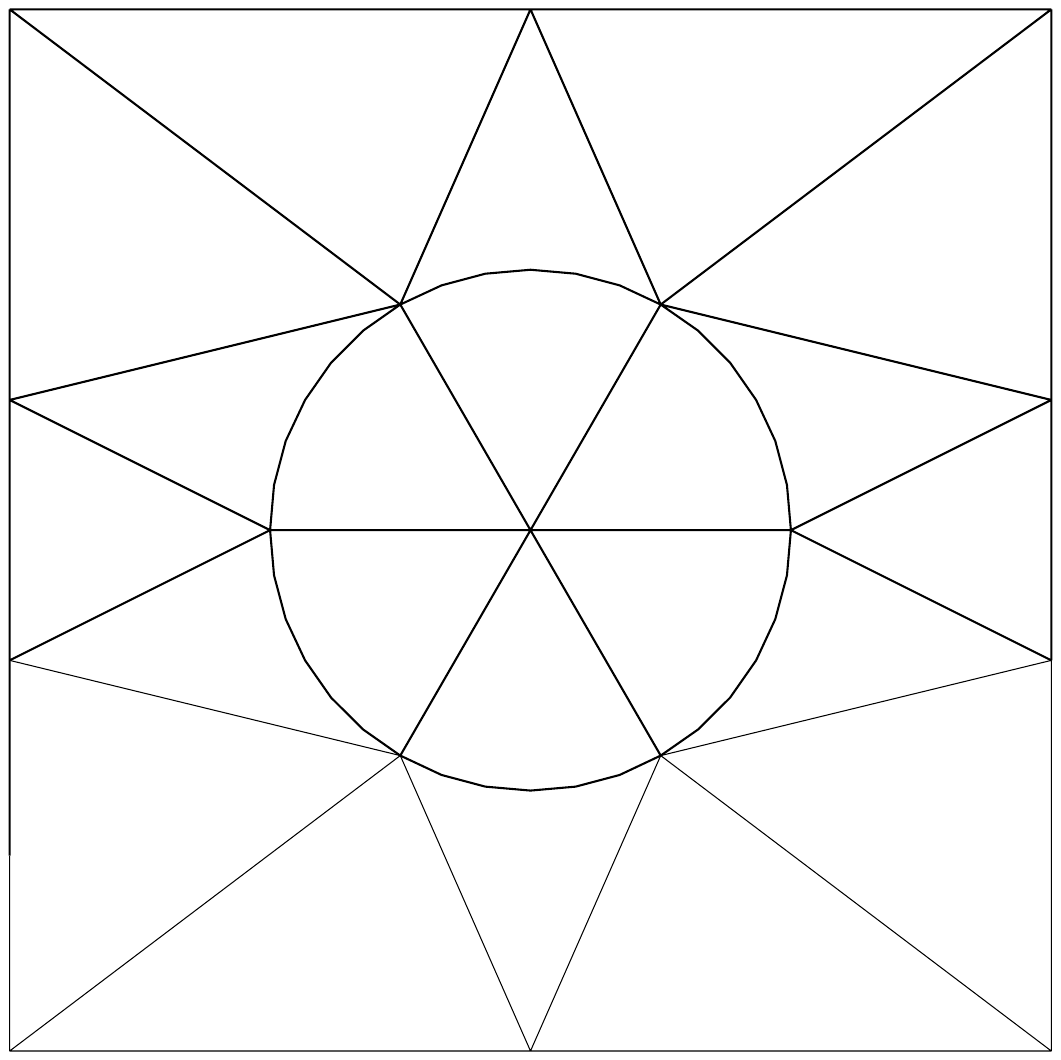}}
 \put(2.02,-0.35){\includegraphics[width=1.4in]{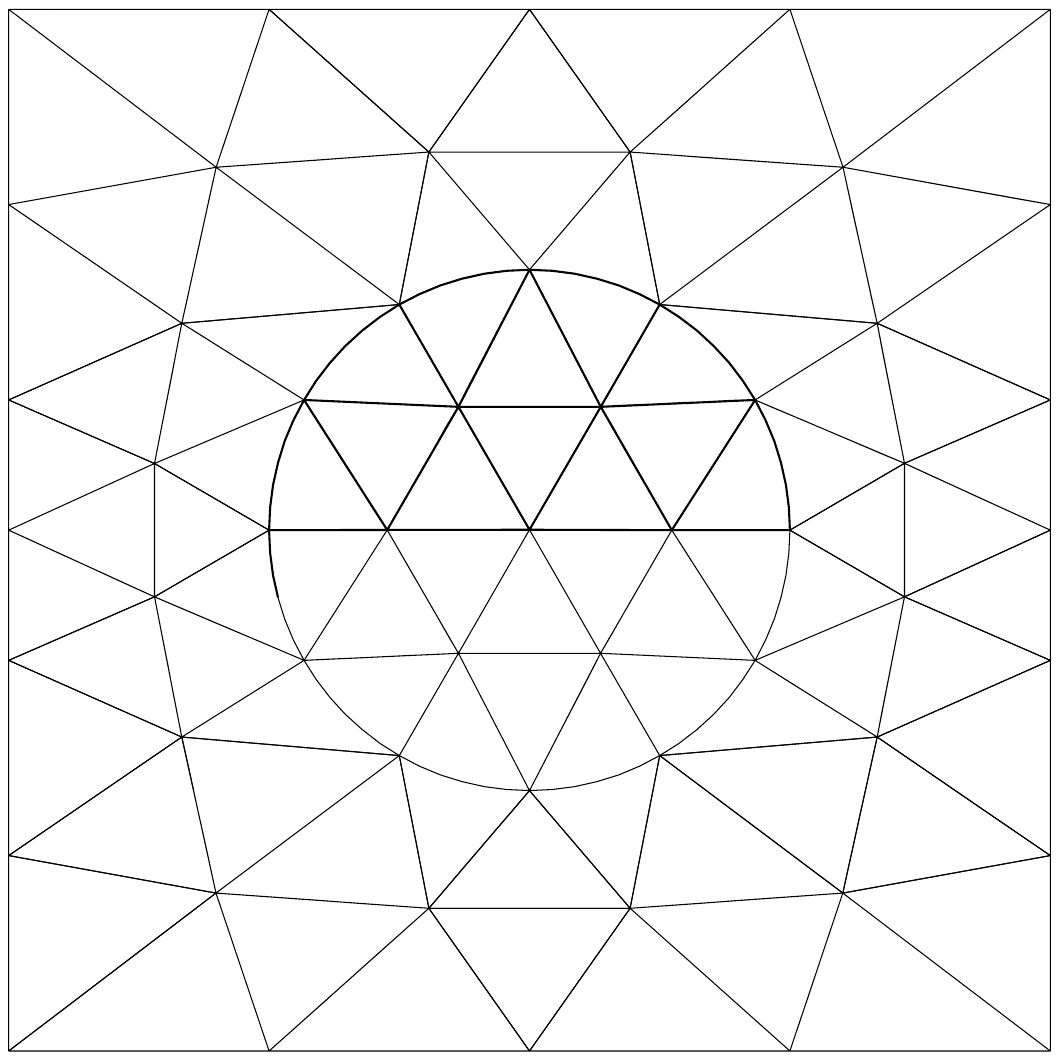}}
 \put(3.03,-0.35){\includegraphics[width=1.4in]{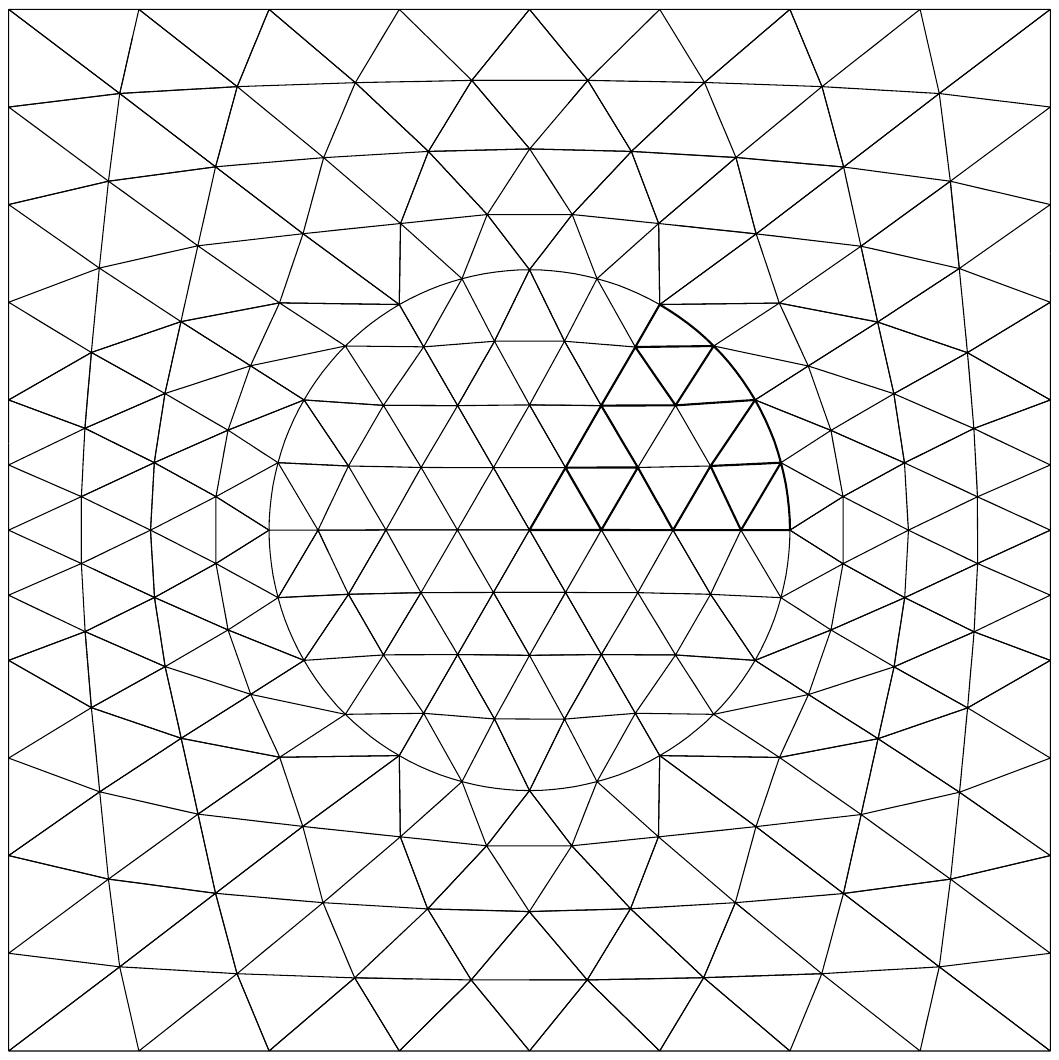}}
    \end{picture}
\caption{The interface $\Gamma$ and the first three grids for the computation in Tables
    \ref{t1}--\ref{t5}.  } \label{g2}
\end{center}
\end{figure}

The weak solution of \eqref{e-1} is 
\an{\label{s-1} 
   u(x,y) = \begin{cases} \mu^{-1} ( 1+\mu -(x^2+y^2)^3) & \t{if }\ x^2+y^2 < 1, \\
                            1 & \t{if }\ x^2+y^2 = 1, \\
                         2-(x^2+y^2)^3 & \t{if }\ x^2+y^2 > 1. \end{cases} }
We note that with the careful construction \eqref{s-1}, the weak solution of \eqref{e-1} is
   the strong solution of \eqref{model-1}--\eqref{model-4} as 
\a{ u|_{\Gamma} \ne 0, \quad  \mu\partial_{\b n}^{(n)} u|_{\Gamma^-}
            =\partial_{\b n}^{(n)} u |_{\Gamma^-}, \ n=1,2,\dots,  }
where the interface $\Gamma=\{ (x,y) : x^2+y^2=1 \}$, $\b n$ is the
   unit outward normal vector on $\Gamma$, and $\partial_{\b n}^{(n)} u $ is the
  $n$-th directional derivative of $u$ in the direction $\b n$.

\begin{table}[ht]
  \centering \renewcommand{\arraystretch}{1.1}
  \caption{The error of $P_1$ elements  for \eqref{e-1} on triangular grids (Figure \ref{g2}) }\label{t1}
\begin{tabular}{c|cc|cc}
\hline
$G_i$  & $\|Q_0 u -  u_0\|_{0,a} $  &rate & $ \|\nabla_w(Q_h u -u_h)\|_{0,a^2} $ &rate    \\
\hline
&\multicolumn{4}{c}{By the $P_1$-$P_{2}$-$P_{2}^2$ finite element, $\mu=10^{-4}$ in \eqref{mu}.}
\\ \hline
 4&  0.1024E-01&4.0&  0.1348E+00&2.9\\ 
 5&  0.6225E-03&4.0&  0.1662E-01&3.0\\ 
 6&  0.3837E-04&4.0&  0.2062E-02&3.0\\ 
 \hline
&\multicolumn{4}{c}{By the $P_1$-$P_{2}$-$P_{2}^2$ finite element, $\mu=1$ in \eqref{mu}.}
\\ \hline
 4&  0.4209E-02&4.0&  0.1348E+00&2.9\\ 
 5&  0.2675E-03&4.0&  0.1733E-01&3.0\\ 
 6&  0.1688E-04&4.0&  0.2200E-02&3.0\\ 
 \hline
&\multicolumn{4}{c}{By the $P_1$-$P_{2}$-$P_{2}^2$ finite element, $\mu=10^{4}$ in \eqref{mu}.}
\\ \hline
 4&  0.1503E+00&4.0&  0.1348E+00&2.9\\ 
 5&  0.9339E-02&4.0&  0.1733E-01&3.0 \\
 6&  0.5956E-03&4.0&  0.2200E-02&3.0\\
 \hline 
\end{tabular}%
\end{table}%

Because of the limited precision of computer double precision algorithm,
we could not reach enough levels of order five or above convergence for the
  $P_k$-$P_{k-1}$-$P_{k-1}^2$ weak Galerkin finite elements.
Instead,  we use a two-order superconvergent $P_k$-$P_{k+1}$-$P_{k+1}^2$ weak Galerkin method, 
 i.e., $v_b|_e\in V_b(e,k+1)$ in \eqref{V-h} and $\nabla_w v =\nabla_{w,k+1,T}$ in
    \eqref{w-g}.
This way, using low degree polynomials,
   we can compute order-eight convergent solutions for the curved-edge interface
  problem \eqref{e-1}.

In Table \ref{t1},  we list the results of the $P_1$-$P_{2}$-$P_{2}^2$ finite element
  for solving the interface problem \eqref{e-1} on meshes shown in Figure \ref{g2}.
Here, to cancel somewhat the difference of solutions with different $\mu$, we
 use a weighted norm to measure the error, 
\a{ \| u \|_{0,a}^2 = \int_{\Omega} a(x,y) u^2(x,y) dx\, dy.  }
Supposedly the $P_1$ finite element converges at order 2 in $L^2$ norm and order 1 in
   $H^1$ norm, respectively.
But as the method of two-order superconvergence, the $P_1$ finite element solution
  converges two orders above the optimal order, in both norms,  in Table \ref{t1}.

\begin{table}[ht]
  \centering \renewcommand{\arraystretch}{1.1}
  \caption{The error of $P_2$ elements  for \eqref{s-1} on triangular grids (Figure \ref{g2}) }\label{t2}
\begin{tabular}{c|cc|cc}
\hline
$G_i$ & $\|Q_0 u -  u_0\|_{0,a} $  &rate & $ \|\nabla_w(Q_h u -u_h)\|_{0,a^2} $ &rate    \\
\hline
&\multicolumn{4}{c}{By the $P_2$-$P_{3}$-$P_{3}^2$ finite element, $\mu=10^{-4}$ in \eqref{mu}.}
\\ \hline
 3&  0.3400E-02&5.0&  0.2364E-01&4.0\\ 
 4&  0.1032E-03&5.0&  0.1406E-02&4.1\\ 
 5&  0.3095E-05&5.1&  0.8373E-04&4.1\\ 
 \hline
&\multicolumn{4}{c}{By the $P_2$-$P_{3}$-$P_{3}^2$ finite element, $\mu=1$ in \eqref{mu}.}
\\ \hline
 3&  0.1092E-02&4.9&  0.2361E-01&4.0\\ 
 4&  0.3545E-04&4.9&  0.1488E-02&4.0\\ 
 5&  0.1141E-05&5.0&  0.9387E-04&4.0\\ 
 \hline
&\multicolumn{4}{c}{By the $P_2$-$P_{3}$-$P_{3}^2$ finite element, $\mu=10^{4}$ in \eqref{mu}.}
\\ \hline
 2&  0.5463E-01&5.2&  0.3779E+00&4.1\\ 
 3&  0.1317E-02&5.4&  0.2363E-01&4.0\\ 
 4&  0.3647E-04&5.2&  0.1488E-02&4.0\\ 
 \hline 
\end{tabular}%
\end{table}%

In Table \ref{t2},  we list the results of the $P_2$-$P_{3}$-$P_{3}^2$ finite element
  for solving the interface problem \eqref{e-1} on meshes shown in Figure \ref{g2}.
The optimal order of convergence of the $P_2$ finite element is
   order 3 and order 2 in $L^2$ norm and $H^1$ norm, respectively.
Here in Table \ref{t2}, the finite element solution converges two orders above the
  optimal order.
It seems from Table \ref{t2} that the error bound is independent of the  
   size of jump of the coefficient $a$ in the interface problem \eqref{e-1}.

\begin{figure}[ht]\begin{center}\setlength\unitlength{1.2in} 
    \begin{picture}(3.03,2.1) 
 \put(0,-2.1){\includegraphics[width=4in]{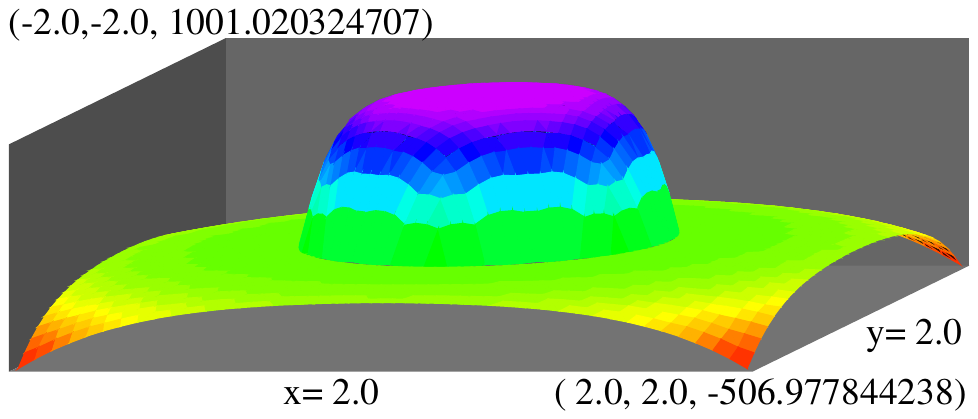}}   
 \put(0,-3.23){\includegraphics[width=4in]{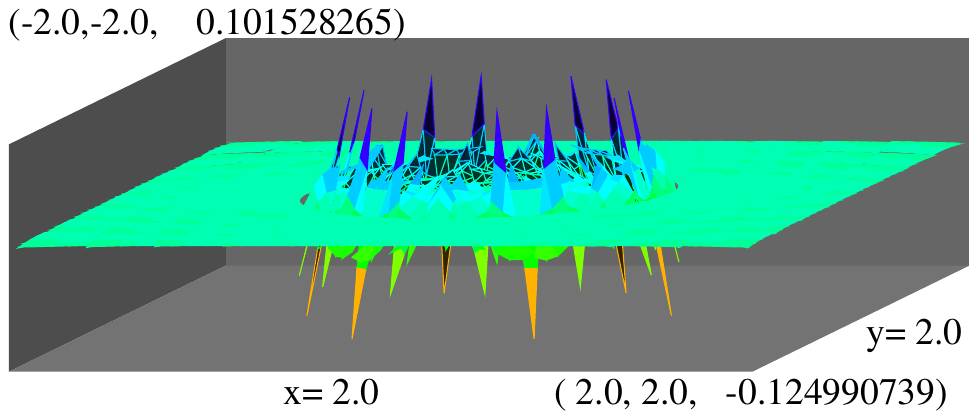}}
    \end{picture}
\caption{Top: \ The $P_2$ finite element solution for \eqref{e-1} with $\mu=10^{-3}$ in
   \eqref{mu} on the third grid $G_3$ in
   Figure \ref{g2}. Bottom: \ The error of the solution above.  } \label{sol2}
\end{center}
\end{figure}

In Figure \ref{sol2}, we plot the $P_2$ solution for the interface problem \eqref{e-1}, where
  $\mu=10^{-3}$, on the third grid $G_3$ in Figure \ref{g2}.
We can see that the normal derivative of the solution jumps to one thousand times large
   at the interface circle, i.e., a sharp turn there.   
Also in Figure \ref{sol2}, we plot the error of above solution on the same mesh.
The error indicates that the method matches the interface curve well and the
  error bound is truly independent of the $\mu$-jump.

\begin{table}[ht]
  \centering \renewcommand{\arraystretch}{1.1}
  \caption{The error of $P_3$ elements  for \eqref{e-1} on triangular grids (Figure \ref{g2}) }\label{t3}
\begin{tabular}{c|cc|cc}
\hline
$G_i$ & $\|Q_0 u -  u_0\|_{0,a} $  &rate & $ \|\nabla_w(Q_h u -u_h)\|_{0,a^2} $ &rate    \\
\hline
&\multicolumn{4}{c}{By the $P_3$-$P_{4}$-$P_{4}^2$ finite element, $\mu=10^{-4}$ in \eqref{mu}.}
\\ \hline
 3&  0.5685E-04&6.1&  0.3073E-03&5.0\\ 
 4&  0.8292E-06&6.1&  0.9668E-05&5.0\\ 
 5&  0.1210E-07&6.1&  0.3021E-06&5.0\\ 
 \hline
&\multicolumn{4}{c}{By the $P_3$-$P_{4}$-$P_{4}^2$ finite element, $\mu=1$ in \eqref{mu}.}
\\ \hline
 3&  0.8999E-05&6.0&  0.3068E-03&5.0\\ 
 4&  0.1441E-06&6.0&  0.9655E-05&5.0\\ 
 5&  0.2360E-08&5.9&  0.3052E-06&5.0\\ 
 \hline
&\multicolumn{4}{c}{By the $P_3$-$P_{4}$-$P_{4}^2$ finite element, $\mu=10^{2}$ in \eqref{mu}.}
\\ \hline
 2&  0.5741E-03&6.2&  0.1006E-01&5.2 \\
 3&  0.8996E-05&6.0&  0.3072E-03&5.0 \\
 4&  0.1441E-06&6.0&  0.9665E-05&5.0 \\
 \hline 
\end{tabular}%
\end{table}%

In Table \ref{t3},  we list the results of the $P_3$-$P_{4}$-$P_{4}^2$ finite element
  for solving the interface problem \eqref{e-1} on meshes shown in Figure \ref{g2}. 
Again in Table \ref{t4}, the finite element solution converges at two orders above the
  optimal order, in both norms.

\begin{table}[ht]
  \centering \renewcommand{\arraystretch}{1.1}
  \caption{The error of $P_4$ elements for \eqref{e-1} on triangular grids (Figure \ref{g2}) }\label{t4}
\begin{tabular}{c|cc|cc}
\hline
$G_i$ & $\|Q_0 u -  u_0\|_{0,a} $  &rate & $ \|\nabla_w(Q_h u -u_h)\|_{0,a^2} $ &rate    \\
\hline
&\multicolumn{4}{c}{By the $P_4$-$P_{5}$-$P_{5}^2$ finite element, $\mu=10^{-4}$ in \eqref{mu}.}
\\ \hline
 2&  0.1010E-03&7.3&  0.8618E-04&6.5\\ 
 3&  0.6374E-06&7.3&  0.1009E-05&6.4\\ 
 4&  0.4795E-08&7.1&  0.1258E-07&6.4\\ 
 \hline
&\multicolumn{4}{c}{By the $P_3$-$P_{4}$-$P_{4}^2$ finite element, $\mu=1$ in \eqref{mu}.}
\\ \hline
 1&  0.3672E-03&0.0&  0.7482E-02&0.0 \\
 2&  0.2404E-05&7.3&  0.8498E-04&6.5 \\
 3&  0.1695E-07&7.1&  0.1042E-05&6.3 \\
 \hline
&\multicolumn{4}{c}{By the $P_3$-$P_{4}$-$P_{4}^2$ finite element, $\mu=10$ in \eqref{mu}.}
\\ \hline
 1&  0.3406E-03&0.0&  0.7541E-02&0.0 \\
 2&  0.2254E-05&7.2&  0.8560E-04&6.5 \\
 3&  0.1612E-07&7.1&  0.1049E-05&6.4 \\
 \hline 
\end{tabular}%
\end{table}%

In Table \ref{t4},  we list the results of the $P_4$-$P_{5}$-$P_{5}^2$ finite element
  for solving the interface problem \eqref{e-1} on meshes shown in Figure \ref{g2}. 
Again in Table \ref{t4}, the finite element solution converges at two orders above the
  optimal order.

Finally, in Table \ref{t5},  we list the results of the $P_5$-$P_{6}$-$P_{6}^2$ finite element
  for solving the interface problem \eqref{e-1} on meshes shown in Figure \ref{g2}. 
The finite element solution converges at order eight, two orders above the
  optimal order, in $L^2$ norm, when $\mu=10^{-1}$. 
But when the error reaches $10^{-9}$ size, the computer accuracy is exhausted that
  we have a slightly less order of convergence at the last level, when $\mu=1$ (smooth
  solution) and $\mu=2$ (a derivative jump solution.)

\begin{table}[ht]
  \centering \renewcommand{\arraystretch}{1.1}
  \caption{The error of $P_5$ elements for \eqref{e-1} on triangular grids (Figure \ref{g2}) }\label{t5}
\begin{tabular}{c|cc|cc}
\hline
$G_i$ & $\|Q_0 u -  u_0\|_{0,a} $  &rate & $ \|\nabla_w(Q_h u -u_h)\|_{0,a^2} $ &rate    \\
\hline
&\multicolumn{4}{c}{By the $P_5$-$P_{6}$-$P_{6}^2$ finite element, $\mu=10^{-4}$ in \eqref{mu}.}
\\ \hline
 1&  0.9320E-03&0.0&  0.5672E-03&0.0\\ 
 2&  0.2808E-05&8.4&  0.3440E-05&7.4\\ 
 3&  0.1028E-07&8.1&  0.3200E-07&6.7\\ 
 \hline
&\multicolumn{4}{c}{By the $P_5$-$P_{6}$-$P_{6}^2$ finite element, $\mu=1$ in \eqref{mu}.}
\\ \hline
 1&  0.2238E-04&0.0&  0.5576E-03&0.0 \\
 2&  0.8150E-07&8.1&  0.3325E-05&7.4 \\
 3&  0.4517E-09&7.5&  0.2190E-07&7.2 \\
 \hline
&\multicolumn{4}{c}{By the $P_5$-$P_{6}$-$P_{6}^2$ finite element, $\mu=2$ in \eqref{mu}.}
\\ \hline
 1&  0.2159E-04&0.0&  0.5585E-03&0.0 \\
 2&  0.7879E-07&8.1&  0.3335E-05&7.4 \\
 3&  0.4194E-09&7.6&  0.2230E-07&7.2 \\
 \hline 
\end{tabular}%
\end{table}%

In the second numerical test,  we solve the interface problem \eqref{weak-formula} with
   a lightly irregular interface curve:
  Find $u\in H^1(\Omega)$ such that $u|_{\partial \Omega} = r^4 (r-3+\cos(4\theta)) $ and
\an{ (a \nabla u, \nabla v) = (48r^2-25r^3, v) \quad\forall v\in H^1_0(\Omega),
   \label{e-2} }
where $r=\sqrt{x^2+y^2}$, $\tan \theta=y/x$ and
   \an{ \label{mu2} \ad{
  a(x,y)&=\begin{cases} \mu & \t{if }\ r < 3-\cos(4\theta), \\
                         1  & \t{if }\ r \ge 3-\cos(4\theta), \end{cases} \\
  \Omega&= (-4,4)\times(-4,4).  } }

\begin{figure}[ht]\begin{center}\setlength\unitlength{1.2in} 
    \begin{picture}(4.03,1.1)
\put(0,0.98){$\Gamma$ and $\partial \Omega$:}
\put(1.02,0.98){$G_1$:}
\put(2.03,0.98){$G_2$:}
\put(3.04,0.98){$G_3$:}
 \put(0,-0.35){\includegraphics[width=1.4in]{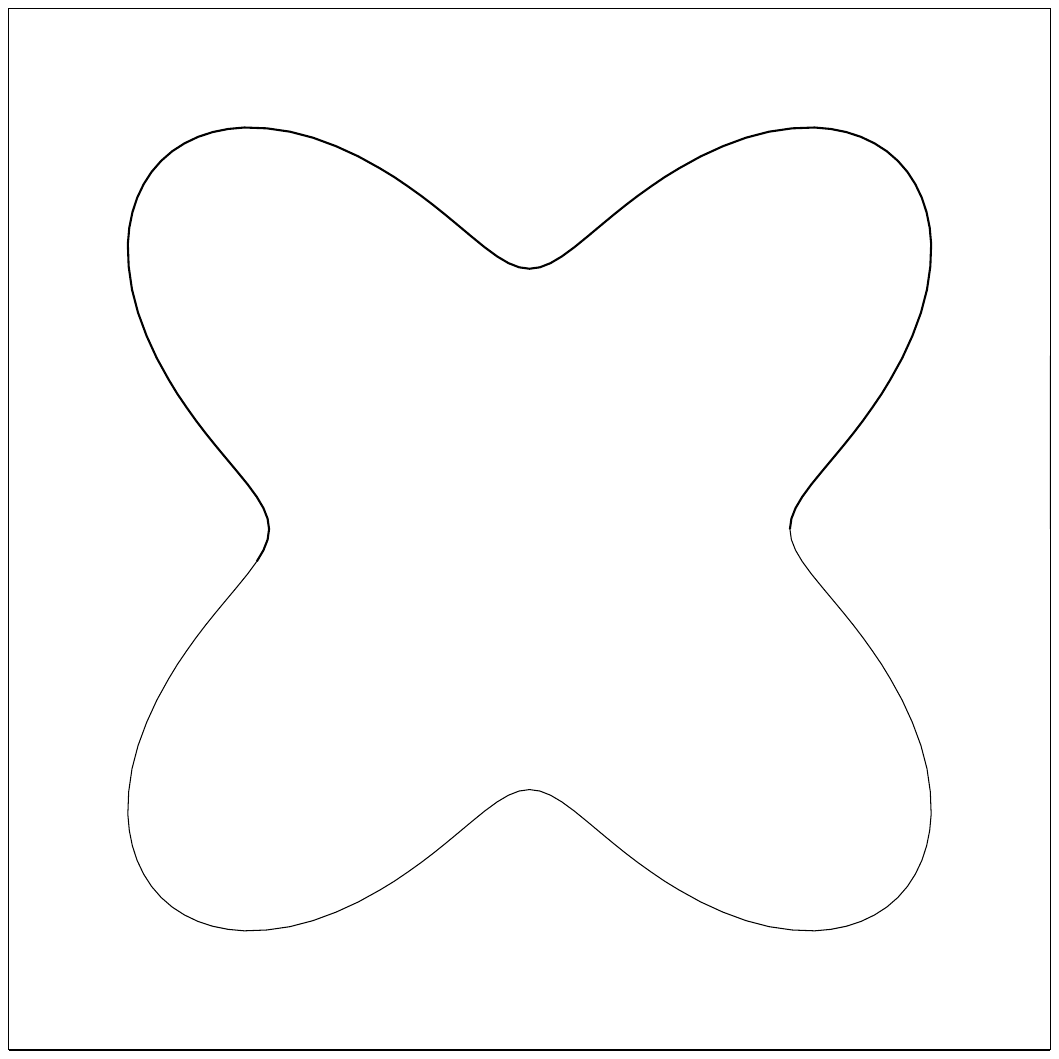}}  
 \put(1.01,-0.35){\includegraphics[width=1.4in]{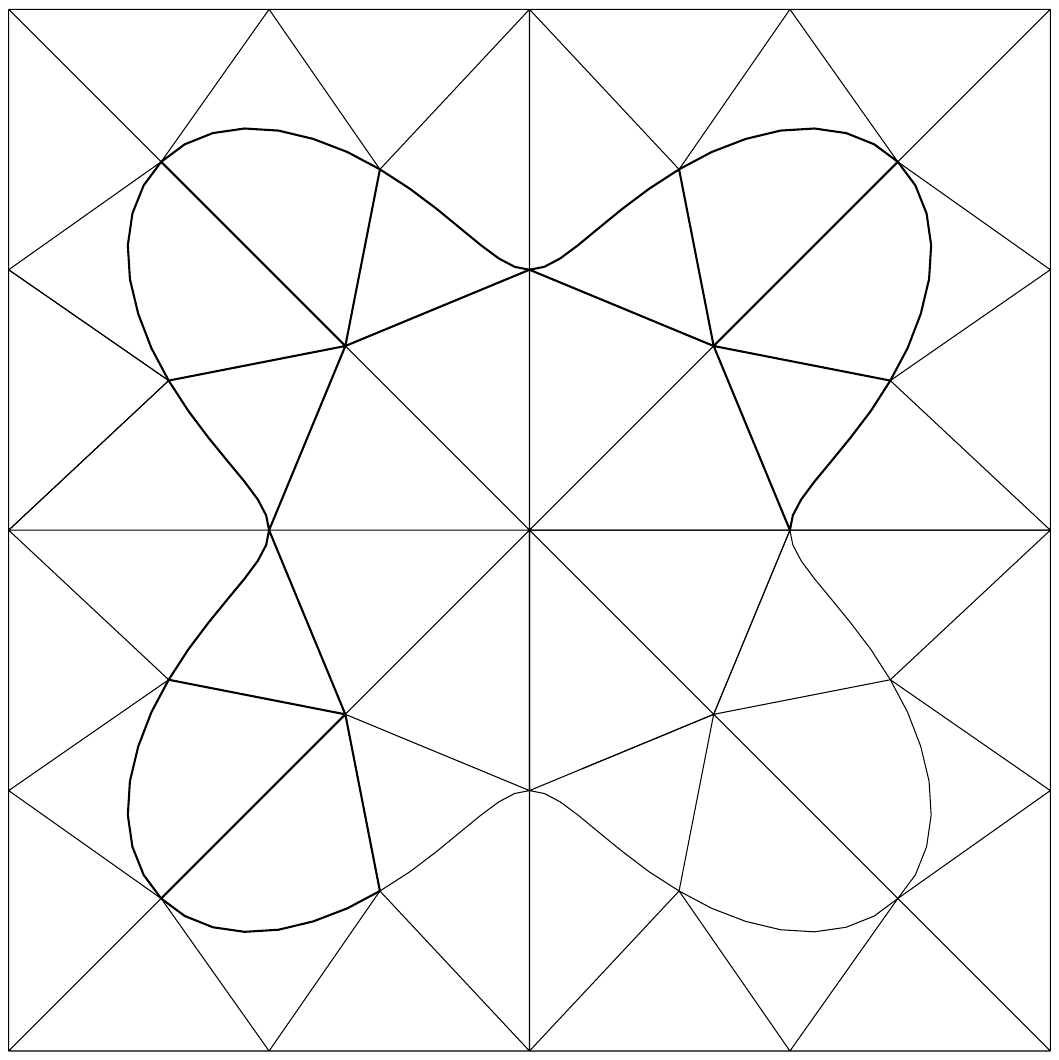}}
 \put(2.02,-0.35){\includegraphics[width=1.4in]{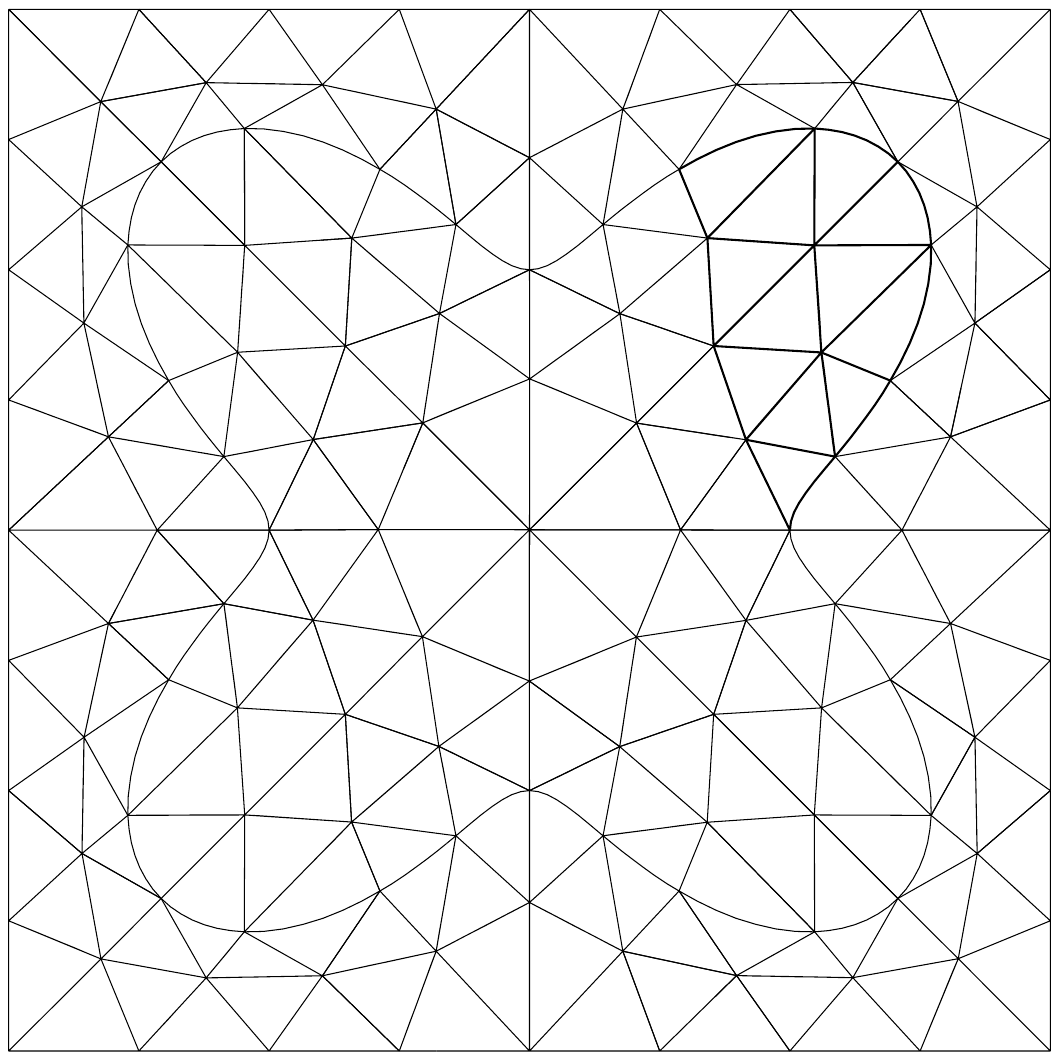}}
 \put(3.03,-0.35){\includegraphics[width=1.4in]{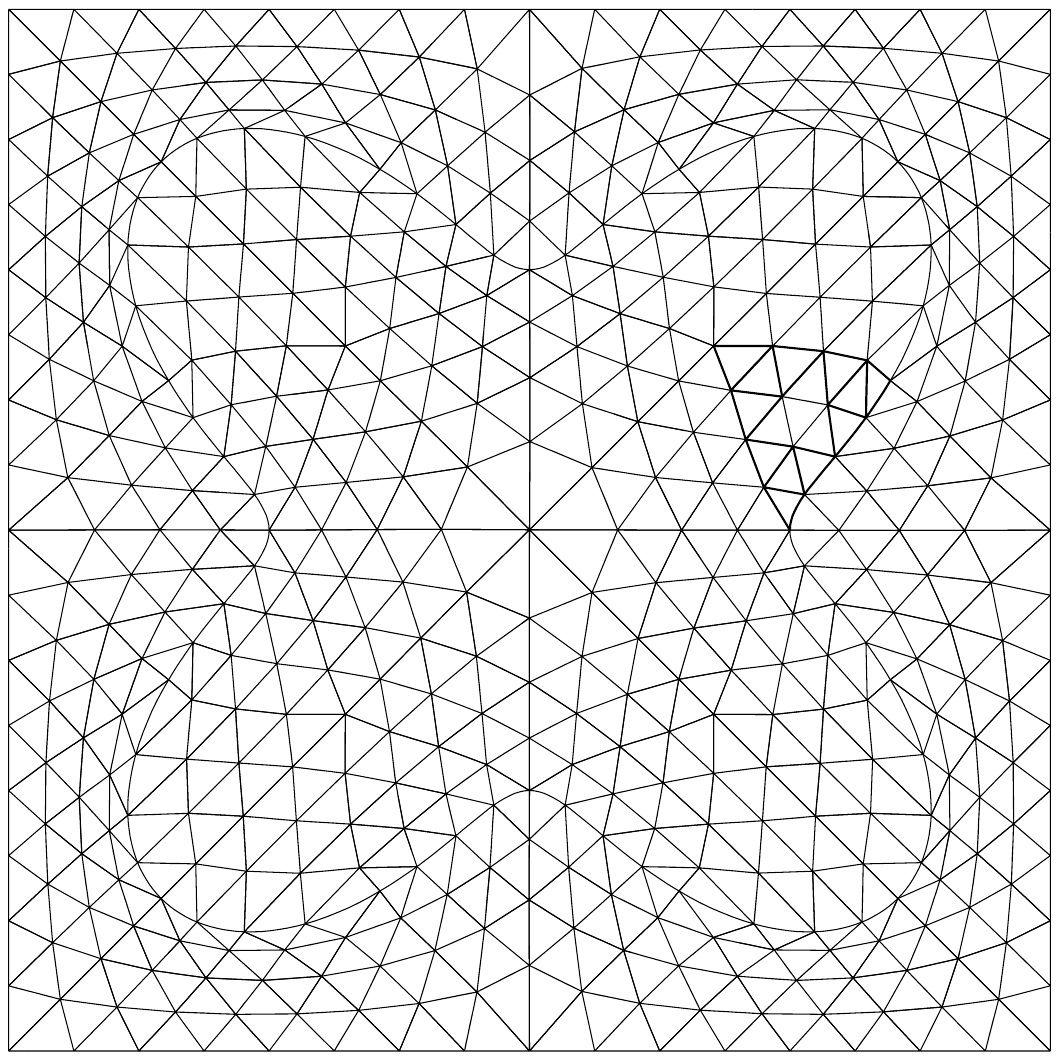}}
    \end{picture}
\caption{The interface $\Gamma$ and the first three grids for the computation in Tables
    \ref{t21}--\ref{t24}.  } \label{g1}
\end{center}
\end{figure}

The weak solution of \eqref{e-2} is 
\an{\label{s-2} 
   u(x,y) = \begin{cases} \mu^{-1} r^4 (r-3+\cos(4\theta)) & \t{if }\ r < 3-\cos(4\theta), \\ 
                   r^4 (r-3+\cos(4\theta)) & \t{if }\ r \ge 3-\cos(4\theta) . \end{cases} } 
The interface is $\Gamma=\{ (x,y) : r = 3-\cos(4\theta) \}$, shown in Figure \ref{g1}.

\begin{table}[ht]
  \centering \renewcommand{\arraystretch}{1.1}
  \caption{The error of $P_1$ elements  for \eqref{s-2} on grids shown in Figure \ref{g1}.}\label{t21}
\begin{tabular}{c|cc|cc}
\hline
$G_i$  & $\|Q_0 u -  u_0\|_{0,a} $  &rate & $ \|\nabla_w(Q_h u -u_h)\|_{0,a^2} $ &rate    \\
\hline
&\multicolumn{4}{c}{By the $P_1$-$P_{2}$-$P_{2}^2$ finite element, $\mu=10^{-2}$ in \eqref{mu2}.}
\\ \hline
 4&  0.5271E-02&4.0&  0.1129E+00&3.0 \\
 5&  0.3371E-03&4.0&  0.1386E-01&3.0 \\
 6&  0.2174E-04&4.0&  0.1716E-02&3.0 \\
 \hline
&\multicolumn{4}{c}{By the $P_1$-$P_{2}$-$P_{2}^2$ finite element, $\mu=1$ in \eqref{mu2}.}
\\ \hline
 4&  0.4209E-02&4.0&  0.1348E+00&2.9\\ 
 5&  0.2675E-03&4.0&  0.1733E-01&3.0\\ 
 6&  0.1688E-04&4.0&  0.2200E-02&3.0\\ 
 \hline
&\multicolumn{4}{c}{By the $P_1$-$P_{2}$-$P_{2}^2$ finite element, $\mu=10^{2}$ in \eqref{mu2}.}
\\ \hline
 4&  0.4062E-02&3.9&  0.1162E+00&3.1 \\
 5&  0.2727E-03&3.9&  0.1408E-01&3.0 \\
 6&  0.1816E-04&3.9&  0.1731E-02&3.0 \\
 \hline 
\end{tabular}%
\end{table}%

In Table  \ref{t21},  we list the computational errors of the $P_1$-$P_{2}$-$P_{2}^2$
   finite element
  for solving the interface problem \eqref{e-2} on meshes shown in Figure \ref{g1}. 
The result is perfect, showing two-order superconvergence, jump-independent error bounds,
  and accurate interface approximation.

\begin{table}[ht]
  \centering \renewcommand{\arraystretch}{1.1}
  \caption{The error of $P_2$ elements  for \eqref{s-2} on grids shown in Figure \ref{g1}.}\label{t22}
\begin{tabular}{c|cc|cc}
\hline
$G_i$  & $\|Q_0 u -  u_0\|_{0,a} $  &rate & $ \|\nabla_w(Q_h u -u_h)\|_{0,a^2} $ &rate    \\
\hline
&\multicolumn{4}{c}{By the $P_2$-$P_{3}$-$P_{3}^2$ finite element, $\mu=10^{-2}$ in \eqref{mu2}.}
\\ \hline
 3&  0.6259E-02&5.1&  0.2519E-01&4.6 \\
 4&  0.2035E-03&4.9&  0.1214E-02&4.4 \\
 5&  0.7258E-05&4.8&  0.6503E-04&4.2 \\
 \hline
&\multicolumn{4}{c}{By the $P_2$-$P_{3}$-$P_{3}^2$ finite element, $\mu=1$ in \eqref{mu2}.}
\\ \hline
 3&  0.9359E-03&5.4&  0.2179E-01&4.5 \\
 4&  0.2705E-04&5.1&  0.1101E-02&4.3 \\
 5&  0.9088E-06&4.9&  0.6097E-04&4.2 \\
 \hline
&\multicolumn{4}{c}{By the $P_2$-$P_{3}$-$P_{3}^2$ finite element, $\mu=10^{2}$ in \eqref{mu2}.}
\\ \hline
 3&  0.5124E-03&5.3&  0.2464E-01&4.6 \\
 4&  0.1606E-04&5.0&  0.1203E-02&4.4 \\
 5&  0.5719E-06&4.8&  0.6470E-04&4.2 \\
 \hline 
\end{tabular}%
\end{table}%

In Table  \ref{t22},  we list the   errors of the $P_2$-$P_{3}$-$P_{3}^2$
   finite element
  for solving the interface problem \eqref{e-2} on meshes shown in Figure \ref{g1}. 
The computation is accurate enough to
   show two-order superconvergence, jump-independent error bounds,
  and accurate interface approximation.

\begin{figure}[ht]\begin{center}\setlength\unitlength{1.2in} 
    \begin{picture}(3.03,2.1) 
 \put(0,-2.1){\includegraphics[width=4in]{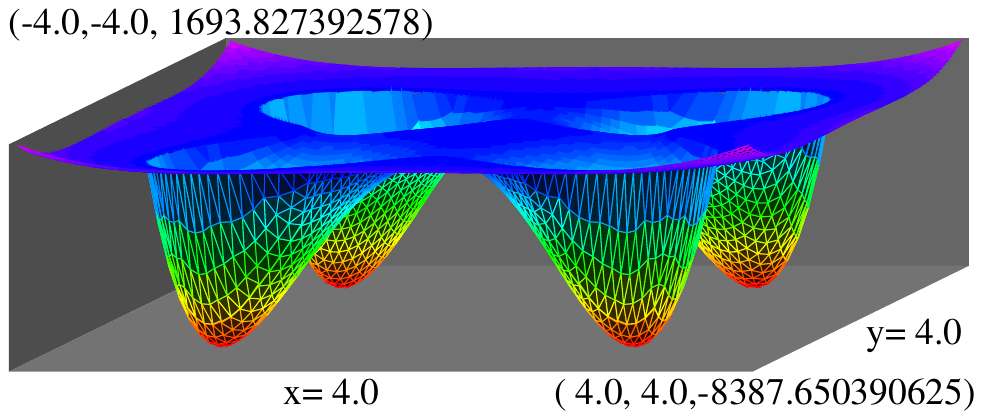}}   
 \put(0,-3.23){\includegraphics[width=4in]{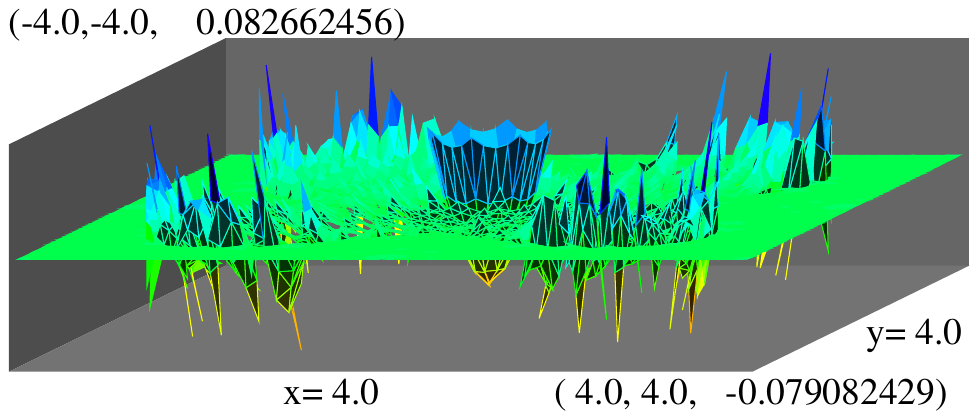}}
    \end{picture}
\caption{Top: \ The $P_2$ finite element solution for \eqref{e-2} with $\mu=10^{-2}$ in
   \eqref{mu2} on the second grid $G_2$ in
   Figure \ref{g1}. Bottom: \ The error of the solution above.  } \label{sol1}
\end{center}
\end{figure}

In Figure \ref{sol1}, the $P_2$ solution for the second interface problem \eqref{e-2}
   is plotted.  
The solution jumps downward at the interface curve.
We can see from the error graph of Figure \ref{sol1} that the error is independent of
   jump-size of the coefficient $\mu$.
It is surprising that the error at the edge of the center hexagon (see the third graph
   in Figure \ref{g1}) is even larger than that at the interface.
In fact, it shows our method approximates the interface very well so that the
   solution error is independent of the interface jump.
On the other side, it shows our method is very accurate that the underline meshes
  must be smooth.
Here, due to the geometry limitation, the meshes changed the pattern near the
   origin that even the regular hexagon and regular triangles are not the
   best mesh shapes there.

\begin{table}[ht]
  \centering \renewcommand{\arraystretch}{1.1}
  \caption{The error of $P_3$ elements  for \eqref{s-2} on grids shown in Figure \ref{g1}.}\label{t23}
\begin{tabular}{c|cc|cc}
\hline
$G_i$  & $\|Q_0 u -  u_0\|_{0,a} $  &rate & $ \|\nabla_w(Q_h u -u_h)\|_{0,a^2} $ &rate    \\
\hline
&\multicolumn{4}{c}{By the $P_3$-$P_{4}$-$P_{4}^2$ finite element, $\mu=10^{-2}$ in \eqref{mu2}.}
\\ \hline
 1&  0.7386E+00&0.0&  0.1385E+01&0.0 \\
 2&  0.1382E-01&5.7&  0.4572E-01&4.9 \\
 3&  0.1904E-03&6.2&  0.1572E-02&4.9 \\
 \hline
&\multicolumn{4}{c}{By the $P_3$-$P_{4}$-$P_{4}^2$ finite element, $\mu=1$ in \eqref{mu2}.}
\\ \hline
 1&  0.1078E+00&0.0&  0.1190E+01&0.0 \\
 2&  0.2541E-02&5.4&  0.4101E-01&4.9 \\
 3&  0.4467E-04&5.8&  0.1349E-02&4.9 \\
 \hline
&\multicolumn{4}{c}{By the $P_3$-$P_{4}$-$P_{4}^2$ finite element, $\mu=10^{2}$ in \eqref{mu2}.}
\\ \hline
 1&  0.5595E-01&0.0&  0.1412E+01&0.0 \\
 2&  0.1201E-02&5.5&  0.4627E-01&4.9 \\
 3&  0.2476E-04&5.6&  0.1615E-02&4.8 \\
 \hline 
\end{tabular}%
\end{table}%

In Table  \ref{t23},  we list the   errors of the $P_3$-$P_{4}$-$P_{4}^2$
   finite element
  for solving the interface problem \eqref{e-2} on meshes shown in Figure \ref{g1}. 
The computation is barely accurate enough to
   show two-order superconvergence, jump-independent error bounds,
  and accurate interface approximation.

\begin{table}[ht]
  \centering \renewcommand{\arraystretch}{1.1}
  \caption{The error of $P_4$ elements  for \eqref{s-2} on grids shown in Figure \ref{g1}.}\label{t24}
\begin{tabular}{c|cc|cc}
\hline
$G_i$  & $\|Q_0 u -  u_0\|_{0,a} $  &rate & $ \|\nabla_w(Q_h u -u_h)\|_{0,a^2} $ &rate    \\
\hline
&\multicolumn{4}{c}{By the $P_4$-$P_{5}$-$P_{5}^2$ finite element, $\mu=2^{-1}$ in \eqref{mu2}.}
\\ \hline
 1&  0.2601E-01&0.0&  0.2876E+00&0.0\\ 
 2&  0.1893E-03&7.1&  0.3930E-02&6.2\\ 
 3&  0.4570E-05&5.4&  0.3638E-04&6.8\\ 
 \hline
&\multicolumn{4}{c}{By the $P_4$-$P_{5}$-$P_{5}^2$ finite element, $\mu=1$ in \eqref{mu2}.}
\\ \hline
 1&  0.2173E-01&0.0&  0.2818E+00&0.0 \\
 2&  0.1599E-03&7.1&  0.3831E-02&6.2 \\
 3&  0.3966E-05&5.3&  0.3579E-04&6.7 \\ 
 \hline
&\multicolumn{4}{c}{By the $P_4$-$P_{5}$-$P_{5}^2$ finite element, $\mu=2$ in \eqref{mu2}.}
\\ \hline
 1&  0.1799E-01&0.0&  0.2876E+00&0.0\\ 
 2&  0.1344E-03&7.1&  0.3941E-02&6.2\\ 
 3&  0.3679E-05&5.2&  0.3643E-04&6.8\\ 
 \hline 
\end{tabular}%
\end{table}%

In Table  \ref{t24},  we list the   errors of the $P_4$-$P_{5}$-$P_{5}^2$
   finite element
  for solving the interface problem \eqref{e-2} on meshes shown in Figure \ref{g1}. 
The computation accuracy is reached that the third level $L^2$ errors do not reach
   the two orders above the optimal order,  due to computer round-off.
It is supposedly of order 7, if computed in a better accuracy computer.

\end{document}